\documentclass[10pt,british]{article}
\usepackage[active]{srcltx}
\usepackage{amsmath,amsfonts,amssymb,mathtools}
\usepackage{latexsym,epsfig}
\usepackage{babel}
\usepackage{theorem}
\usepackage[latin1]{inputenc}
\usepackage{microtype}
\usepackage{fancyhdr} 
\usepackage{lmodern}
\usepackage{color}
\usepackage{enumerate}
\usepackage{amsbsy}
\usepackage[all]{xy}
\usepackage{endnotes}

\usepackage{graphicx}

\usepackage{url}

\usepackage{cite}
	
\renewcommand{\epsilon}{\ensuremath{\varepsilon}}
\renewcommand{\phi}{\ensuremath{\varphi}}
\renewcommand{\to}{\ensuremath{\longrightarrow}}

\newcommand{\K}{\ensuremath{\mathbb K}}

\newcommand{\N}{\ensuremath{\mathbb N}}
\newcommand{\R}{\ensuremath{\mathbb R}}
\newcommand{\Z}{\ensuremath{\mathbb Z}}

\newcommand{\C}{\ensuremath{\mathbb C}}

\newcommand{\FF}{\ensuremath{\mathbb F}}

\renewcommand{\ker}[1]{\ensuremath{\operatorname{\textnormal{Ker}}({#1})}}
\newcommand{\im}[1]{\ensuremath{\operatorname{\textnormal{Im}}({#1})}}

\makeatletter

\renewcommand{\p@enumii}{}

\def\@enum@{\list{\csname label\@enumctr\endcsname}%
           {\usecounter{\@enumctr}\def\makelabel##1{
\normalfont\ignorespaces\emph{{##1}~}}
\setlength{\labelsep}{3pt}
\setlength{\parsep}{0pt}
\setlength{\itemsep}{5pt}
\setlength{\leftmargin}{0pt}
\setlength{\labelwidth}{0pt}
\setlength{\listparindent}{\parindent}
\setlength{\itemindent}{0pt}
\setlength{\topsep}{3pt plus 1pt minus 1 pt}}}

\def\@map#1#2[#3]{\mbox{$#1 \colon\thinspace #2 \to #3$}}
\def\map#1#2{\@ifnextchar [{\@map{#1}{#2}}{\@map{#1}{#2}[#2]}}

\DeclareRobustCommand*\textsubscript[1]{\@textsubscript{\selectfont#1}}
\def\@textsubscript#1{{\m@th\ensuremath{_{\mbox{\fontsize\sf@size\z@#1}}}}}
\makeatother

\DeclareMathOperator{\id}{\textnormal{Id}}

\theoremheaderfont{\normalfont\scshape}

\newtheorem{thm}{Theorem}
\newtheorem{thmAnnexe}{Theorem}
\newtheorem{lem}{Lemma}
\newtheorem{prop}{Proposition}
\newtheorem{cor}{Corollary}
\newtheorem{deft}{Definition}

\theorembodyfont{\rmfamily}

\newcommand{\eop}{%
  \relax
  \ifvmode
    \noindent
  \else
    \unskip
    \hskip0pt plus-1fill\relax 
  \fi
  \vrule width0pt
  \nobreak
  \hfill 
  {\hspace*{\fill}$\Box$}
}

\newenvironment{proof}{\par\vspace{\partopsep}\noindent\emph{Proof.}}
{\eop\par\vspace{\parsep}}
\newenvironment{prooftext}[1]{\par\vspace{\parsep}\noindent{\emph{Proof of #1.}}}
{\eop\par\vspace{\parsep}}

\newtheorem{rem}{Remark}

\newcommand{\gam}[2][n]{\gamma_{#1}^{p}#2}

\newcommand{\build}[3]{\mathrel{\mathop{\kern 0pt#1}\limits_{#2}^{#3}}}

\newcounter{liste}\newcounter{listes}
\newcommand{\svte}{\theliste\addtocounter{liste}{1}}
\newcommand{\eqsvte}{\alph{listes}\addtocounter{listes}{1}}
\newcommand{\eqprec}{\addtocounter{listes}{-1}\alph{listes}\addtocounter{listes}{1}}

\makeatletter
\renewcommand{\@makeenmark}{\hbox{\,\@theenmark}}
\renewcommand{\enoteformat}{\parindent=1.5em\leavevmode\llap{\hbox{\bf \@theenmark}}}
\makeatother

\allowdisplaybreaks

\newlength{\longueur}\newlength{\longueurA}\newlength{\longueurB}

\newcommand{\GenerateursTorelli}{%
\begin{picture}(1,0.64335536)%
    \put(0,0){\includegraphics[width=\unitlength]{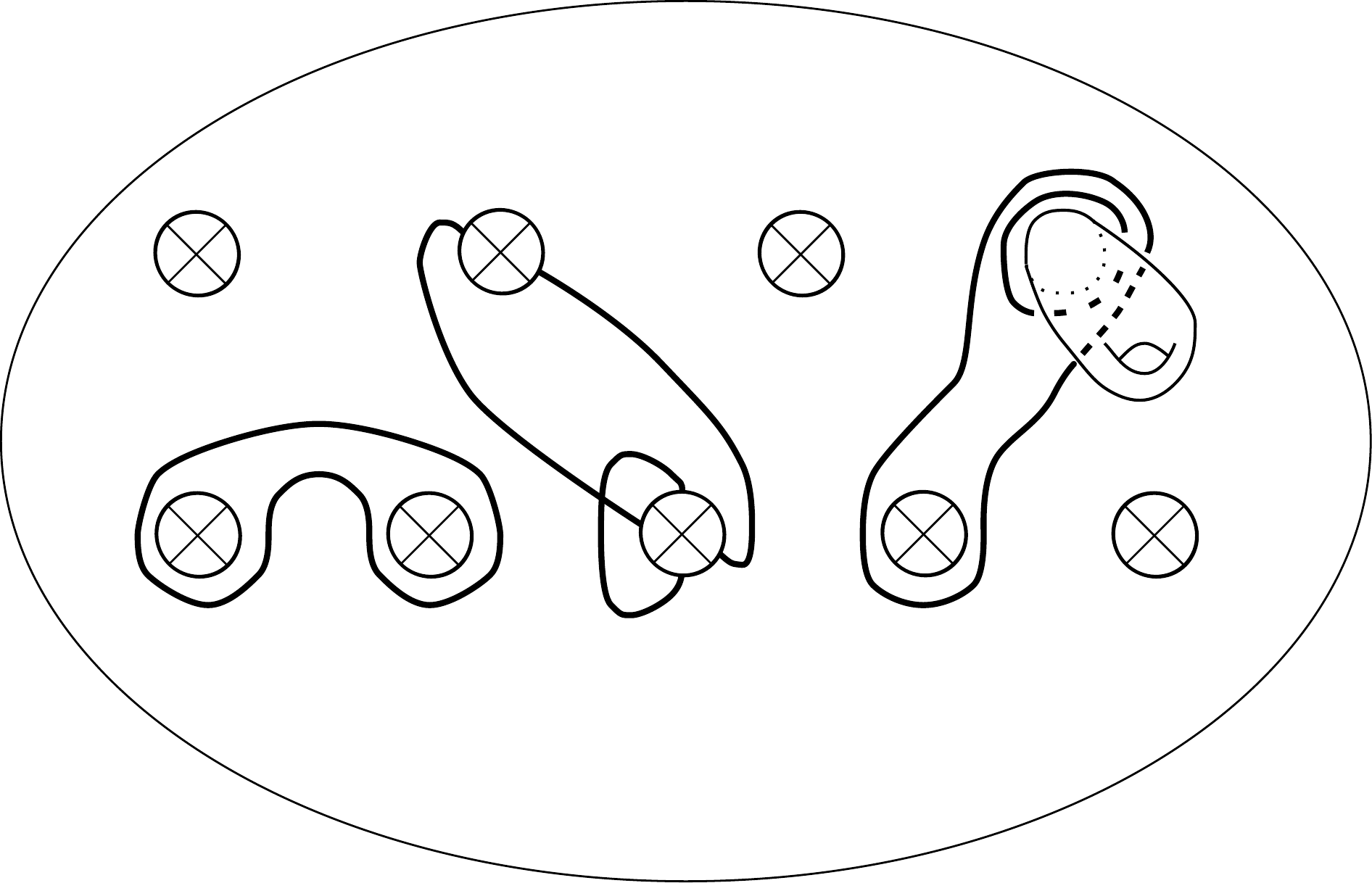}}%
    \put(0.43075439,0.45174422){\color[rgb]{0,0,0}\makebox(0,0)[lb]{\smash{.  .  .}}}%
    \put(0.20024017,0.25052588){\color[rgb]{0,0,0}\makebox(0,0)[lb]{\smash{. . .}}}%
    \put(0.37201725,0.25350634){\color[rgb]{0,0,0}\makebox(0,0)[lb]{\smash{. . .}}}%
    \put(0.5550905,0.24987161){\color[rgb]{0,0,0}\makebox(0,0)[lb]{\smash{. . .}}}%
    \put(0.72633019,0.24987161){\color[rgb]{0,0,0}\makebox(0,0)[lb]{\smash{. . .}}}%
    \put(0.46012298,0.40769136){\color[rgb]{0,0,0}\makebox(0,0)[lb]{\smash{$\alpha_{k,l}$}}}%
    \put(0.12868815,0.50900029){\color[rgb]{0,0,0}\makebox(0,0)[lb]{\smash{$1$}}}%
    \put(0.35824061,0.51105482){\color[rgb]{0,0,0}\makebox(0,0)[lb]{\smash{$l$}}}%
    \put(0.57025514,0.51048138){\color[rgb]{0,0,0}\makebox(0,0)[lb]{\smash{$g$}}}%
    \put(0.12238436,0.15805847){\color[rgb]{0,0,0}\makebox(0,0)[lb]{\smash{$1$}}}%
    \put(0.29859581,0.15805847){\color[rgb]{0,0,0}\makebox(0,0)[lb]{\smash{$j$}}}%
    \put(0.47480727,0.15805847){\color[rgb]{0,0,0}\makebox(0,0)[lb]{\smash{$k$}}}%
    \put(0.65101872,0.15805847){\color[rgb]{0,0,0}\makebox(0,0)[lb]{\smash{$u$}}}%
    \put(0.82723018,0.15805847){\color[rgb]{0,0,0}\makebox(0,0)[lb]{\smash{$n$}}}%
    \put(0.21049009,0.45174422){\color[rgb]{0,0,0}\makebox(0,0)[lb]{\smash{.  .  .}}}%
    \put(0.62165014,0.45174422){\color[rgb]{0,0,0}\makebox(0,0)[lb]{\smash{.  .  .}}}%
    \put(0.87128303,0.45174422){\color[rgb]{0,0,0}\makebox(0,0)[lb]{\smash{.  .  .}}}%
    \put(0.76849302,0.52516566){\color[rgb]{0,0,0}\makebox(0,0)[lb]{\smash{$t$}}}%
    \put(0.40138583,0.18742704){\color[rgb]{0,0,0}\makebox(0,0)[lb]{\smash{$\mu_{k}$}}}%
    \put(0.74735864,0.30490135){\color[rgb]{0,0,0}\makebox(0,0)[lb]{\smash{$\xi_{u,t}$}}}%
    \put(0.0783315,0.33426992){\color[rgb]{0,0,0}\makebox(0,0)[lb]{\smash{$\beta_{1,j}$}}}%
    \put(0.71709802,0.39300707){\color[rgb]{0,0,0}\makebox(0,0)[lb]{\smash{$\delta_{t}$}}}%
\end{picture}%
}

\newcommand{\Generateurs}{%
\begin{picture}(1,0.64335536)%
    \put(0,0){\includegraphics[width=\unitlength]{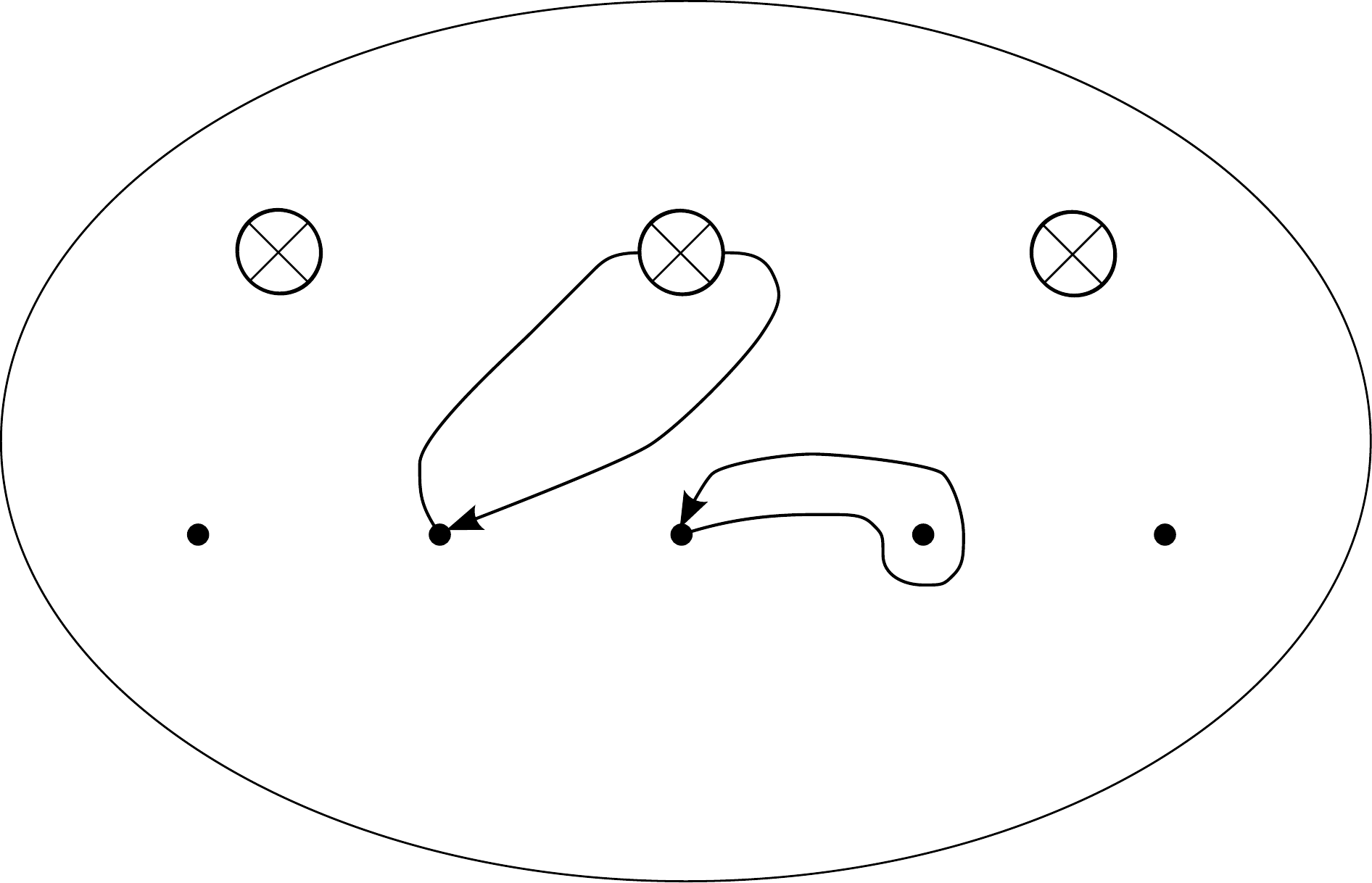}}%
    \put(0.28391153,0.45174423){\color[rgb]{0,0,0}\makebox(0,0)[lb]{\smash{.  .  .}}}%
    \put(0.57759728,0.45174423){\color[rgb]{0,0,0}\makebox(0,0)[lb]{\smash{.  .  .}}}%
    \put(0.20024017,0.25052588){\color[rgb]{0,0,0}\makebox(0,0)[lb]{\smash{. . .}}}%
    \put(0.37376195,0.25059857){\color[rgb]{0,0,0}\makebox(0,0)[lb]{\smash{. . .}}}%
    \put(0.54822871,0.24987161){\color[rgb]{0,0,0}\makebox(0,0)[lb]{\smash{. . .}}}%
    \put(0.72633019,0.24987161){\color[rgb]{0,0,0}\makebox(0,0)[lb]{\smash{. . .}}}%
    \put(0.22517438,0.32692778){\color[rgb]{0,0,0}\makebox(0,0)[lb]{\smash{$\rho_{k,l}$}}}%
    \put(0.18846366,0.51048138){\color[rgb]{0,0,0}\makebox(0,0)[lb]{\smash{$1$}}}%
    \put(0.48949156,0.51048138){\color[rgb]{0,0,0}\makebox(0,0)[lb]{\smash{$l$}}}%
    \put(0.76849302,0.51048138){\color[rgb]{0,0,0}\makebox(0,0)[lb]{\smash{$g$}}}%
    \put(0.12972651,0.1800849){\color[rgb]{0,0,0}\makebox(0,0)[lb]{\smash{$1$}}}%
    \put(0.30593796,0.1800849){\color[rgb]{0,0,0}\makebox(0,0)[lb]{\smash{$k$}}}%
    \put(0.48214941,0.1800849){\color[rgb]{0,0,0}\makebox(0,0)[lb]{\smash{$i$}}}%
    \put(0.65836087,0.1800849){\color[rgb]{0,0,0}\makebox(0,0)[lb]{\smash{$j$}}}%
    \put(0.83457232,0.1800849){\color[rgb]{0,0,0}\makebox(0,0)[lb]{\smash{$n$}}}%
    \put(0.67304515,0.31224349){\color[rgb]{0,0,0}\makebox(0,0)[lb]{\smash{$B_{i,j}$}}}%
\end{picture}
}

\newcommand{\identityE}{%
\begin{picture}(1,0.35406927)%
    \put(0,0){\includegraphics[width=\unitlength]{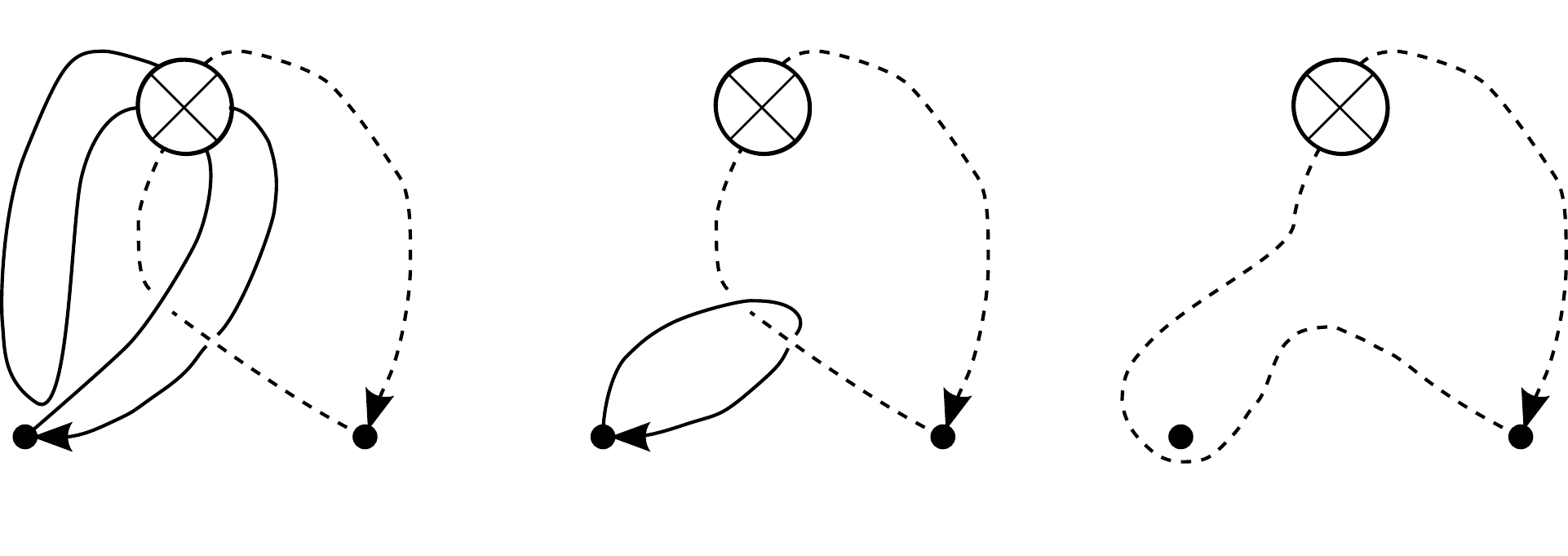}}%
    \put(0.80378153,0.00340892){\color[rgb]{0,0,0}\makebox(0,0)[lb]{\smash{$B_{i,j}^{-1}\rho_{j,k}$}}}%
    \put(0.03769835,0.00340892){\color[rgb]{0,0,0}\makebox(0,0)[lb]{\smash{$\rho_{i,k}^{-1}\rho_{j,k}\rho_{i,k}$}}}%
    \put(0.00156236,0.03954492){\color[rgb]{0,0,0}\makebox(0,0)[lb]{\smash{$i$}}}%
    \put(0.21837835,0.03954492){\color[rgb]{0,0,0}\makebox(0,0)[lb]{\smash{$j$}}}%
    \put(0.09551595,0.07568092){\color[rgb]{0,0,0}\makebox(0,0)[lb]{\smash{. . .}}}%
    \put(0.37014954,0.03954492){\color[rgb]{0,0,0}\makebox(0,0)[lb]{\smash{$i$}}}%
    \put(0.58696553,0.03954492){\color[rgb]{0,0,0}\makebox(0,0)[lb]{\smash{$j$}}}%
    \put(0.46410314,0.07568092){\color[rgb]{0,0,0}\makebox(0,0)[lb]{\smash{. . .}}}%
    \put(0.32678634,0.16963452){\color[rgb]{0,0,0}\makebox(0,0)[lb]{\smash{=}}}%
    \put(0.73873673,0.03954492){\color[rgb]{0,0,0}\makebox(0,0)[lb]{\smash{$i$}}}%
    \put(0.95555272,0.03954492){\color[rgb]{0,0,0}\makebox(0,0)[lb]{\smash{$j$}}}%
    \put(0.83269033,0.07568092){\color[rgb]{0,0,0}\makebox(0,0)[lb]{\smash{. . .}}}%
    \put(0.66646473,0.16963452){\color[rgb]{0,0,0}\makebox(0,0)[lb]{\smash{=}}}%
    \put(0.10997035,0.34308731){\color[rgb]{0,0,0}\makebox(0,0)[lb]{\smash{$k$}}}%
    \put(0.47133034,0.34308731){\color[rgb]{0,0,0}\makebox(0,0)[lb]{\smash{$k$}}}%
    \put(0.83558121,0.34308731){\color[rgb]{0,0,0}\makebox(0,0)[lb]{\smash{$k$}}}%
\end{picture}
}
  
\newcommand{\identityG}{%
\begin{picture}(1,0.64335536)%
    \put(0,0){\includegraphics[width=\unitlength]{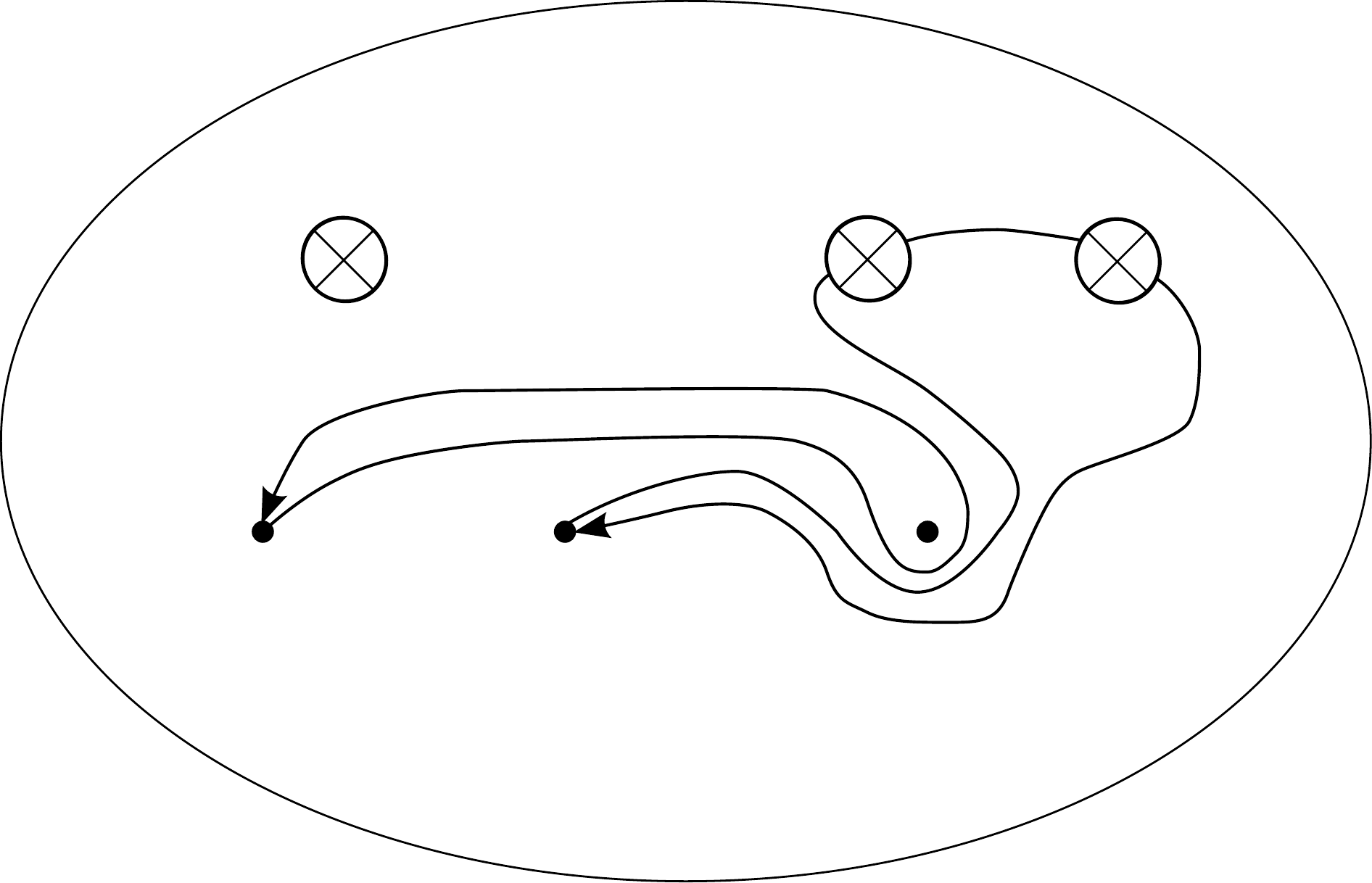}}%
    \put(0.66069427,0.15994446){\color[rgb]{0,0,0}\makebox(0,0)[lb]{\smash{$k$}}}%
    \put(0.176926,0.16749835){\color[rgb]{0,0,0}\makebox(0,0)[lb]{\smash{$i$}}}%
    \put(0.39719031,0.16602992){\color[rgb]{0,0,0}\makebox(0,0)[lb]{\smash{$j$}}}%
    \put(0.27237387,0.25560408){\color[rgb]{0,0,0}\makebox(0,0)[lb]{\smash{. . .}}}%
    \put(0.49263818,0.25560408){\color[rgb]{0,0,0}\makebox(0,0)[lb]{\smash{. . .}}}%
    \put(0.77163965,0.25560408){\color[rgb]{0,0,0}\makebox(0,0)[lb]{\smash{. . .}}}%
    \put(0.61011249,0.50523697){\color[rgb]{0,0,0}\makebox(0,0)[lb]{\smash{$g-1$}}}%
    \put(0.80100823,0.50523697){\color[rgb]{0,0,0}\makebox(0,0)[lb]{\smash{$g$}}}%
    \put(0.41921674,0.38042052){\color[rgb]{0,0,0}\makebox(0,0)[lb]{\smash{$B_{i,k}$}}}%
    \put(0.23881316,0.50523697){\color[rgb]{0,0,0}\makebox(0,0)[lb]{\smash{$1$}}}%
    \put(0.66150749,0.38042052){\color[rgb]{0,0,0}\makebox(0,0)[lb]{\smash{$a_{k}a_{j}a_{k}^{-1}$}}}%
    \put(0.08882027,0.25560408){\color[rgb]{0,0,0}\makebox(0,0)[lb]{\smash{. . .}}}%
    \put(0.38984817,0.44649982){\color[rgb]{0,0,0}\makebox(0,0)[lb]{\smash{. . .}}}%
\end{picture}
}

\newcommand{\identityK}{%
\begin{picture}(1,0.64335536)%
    \put(0,0){\includegraphics[width=\unitlength]{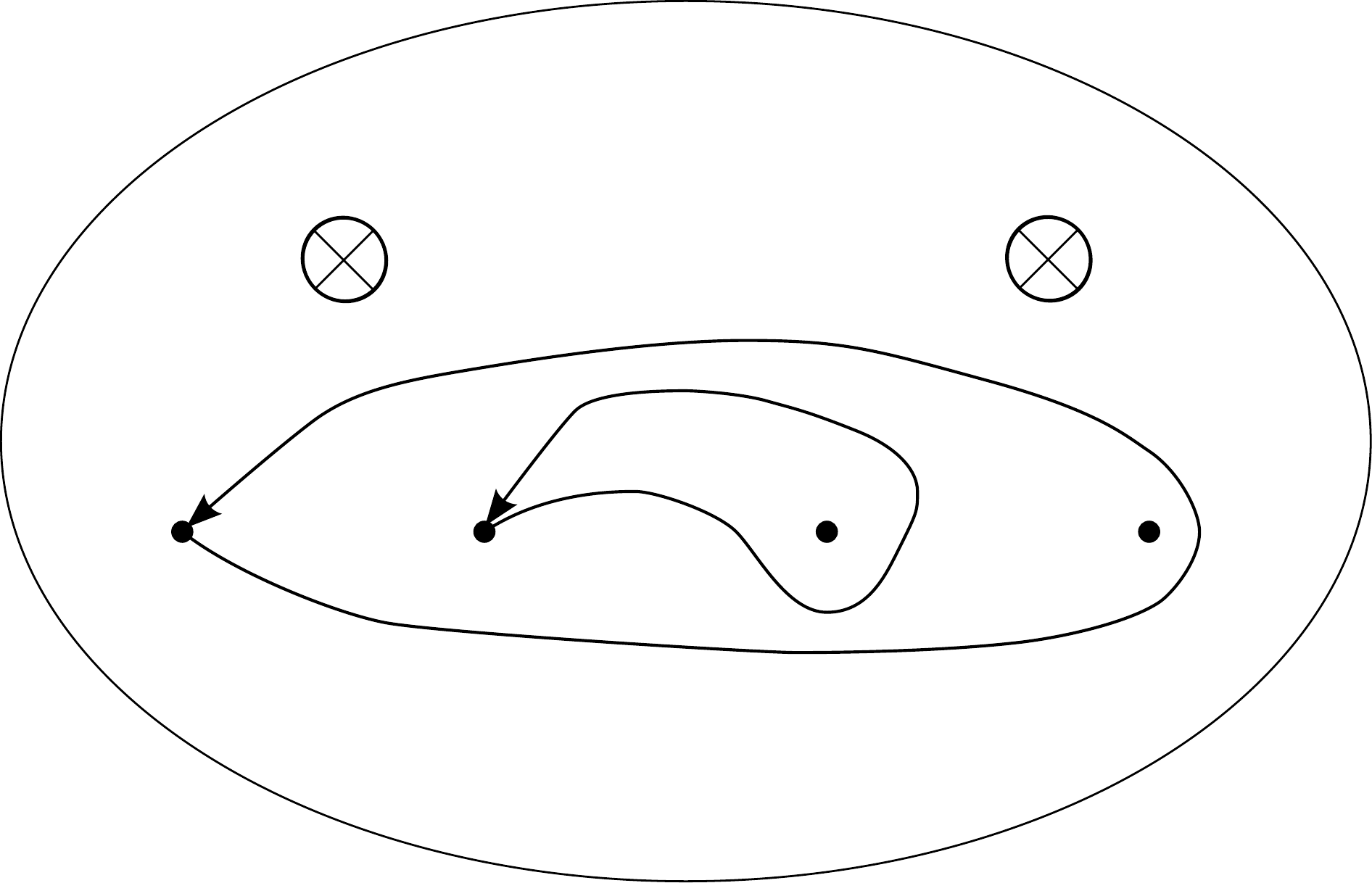}}%
    \put(0.58808605,0.21889336){\color[rgb]{0,0,0}\makebox(0,0)[lb]{\smash{$k$}}}%
    \put(0.11818884,0.21889336){\color[rgb]{0,0,0}\makebox(0,0)[lb]{\smash{$1$}}}%
    \put(0.33845316,0.21889336){\color[rgb]{0,0,0}\makebox(0,0)[lb]{\smash{$j$}}}%
    \put(0.74847569,0.50523697){\color[rgb]{0,0,0}\makebox(0,0)[lb]{\smash{$g$}}}%
    \put(0.4984346,0.28980299){\color[rgb]{0,0,0}\makebox(0,0)[lb]{\smash{$B_{j,k}$}}}%
    \put(0.23881316,0.50523697){\color[rgb]{0,0,0}\makebox(0,0)[lb]{\smash{$1$}}}%
    \put(0.46326961,0.44649982){\color[rgb]{0,0,0}\makebox(0,0)[lb]{\smash{. . .}}}%
    \put(0.21363672,0.25560408){\color[rgb]{0,0,0}\makebox(0,0)[lb]{\smash{. . .}}}%
    \put(0.43390103,0.25560408){\color[rgb]{0,0,0}\makebox(0,0)[lb]{\smash{. . .}}}%
    \put(0.69821821,0.25560408){\color[rgb]{0,0,0}\makebox(0,0)[lb]{\smash{. . .}}}%
    \put(0.82303466,0.21889336){\color[rgb]{0,0,0}\makebox(0,0)[lb]{\smash{$n$}}}%
    \put(0.7735718,0.3612923){\color[rgb]{0,0,0}\makebox(0,0)[lb]{\smash{$T_{1}$}}}%
\end{picture}
}

\newcommand{\GenerateursBord}{%
\begin{picture}(1,0.64335536)%
    \put(0,0){\includegraphics[width=\unitlength]{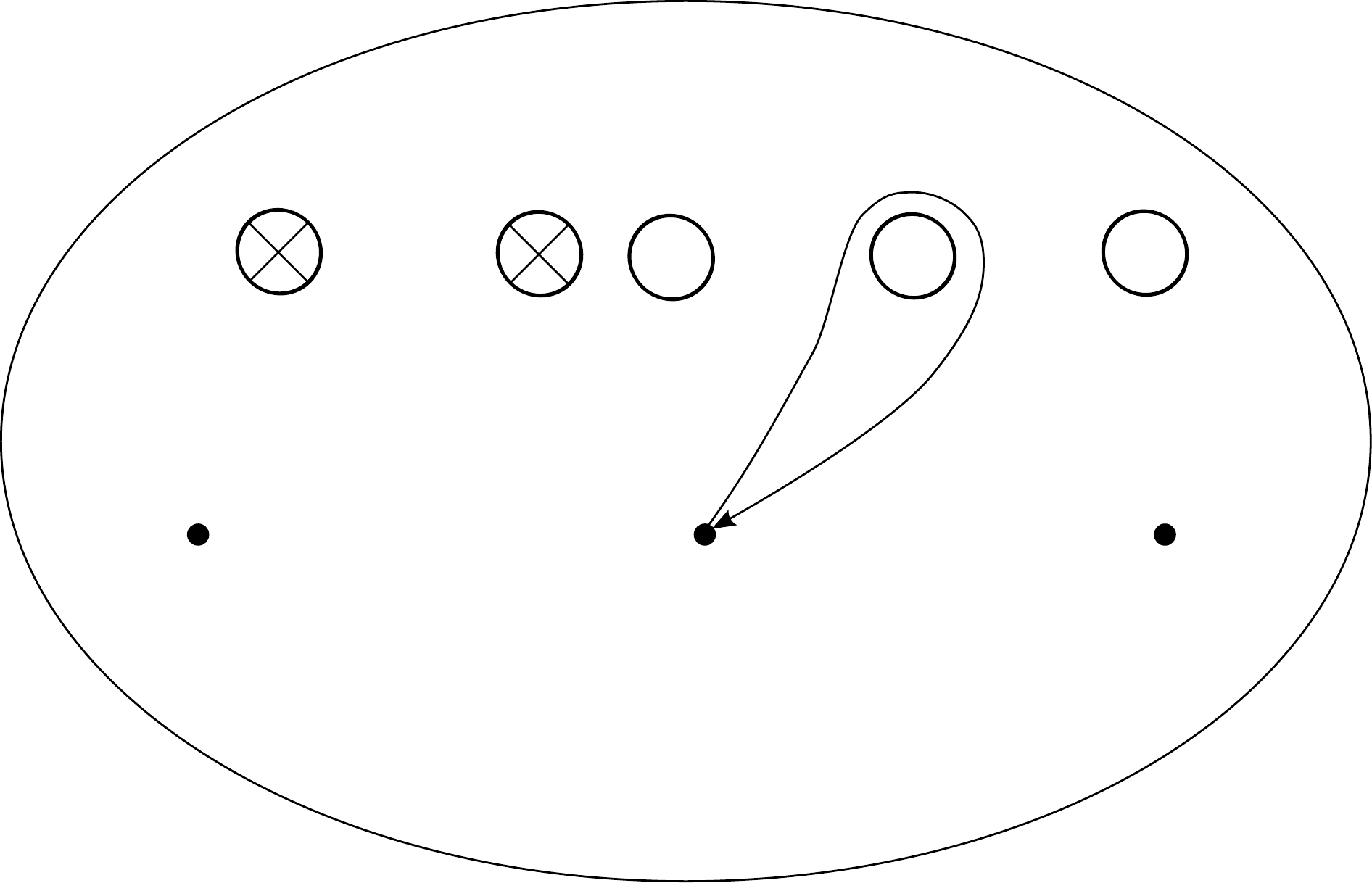}}%
    \put(0.3202935,0.25054757){\color[rgb]{0,0,0}\makebox(0,0)[lb]{\smash{. . .}}}%
    \put(0.66022378,0.24945171){\color[rgb]{0,0,0}\makebox(0,0)[lb]{\smash{. . .}}}%
    \put(0.66575029,0.33482332){\color[rgb]{0,0,0}\makebox(0,0)[lb]{\smash{$x_{k,t}$}}}%
    \put(0.18846366,0.51048138){\color[rgb]{0,0,0}\makebox(0,0)[lb]{\smash{$1$}}}%
    \put(0.37640063,0.51048138){\color[rgb]{0,0,0}\makebox(0,0)[lb]{\smash{$g$}}}%
    \put(0.13163417,0.20303605){\color[rgb]{0,0,0}\makebox(0,0)[lb]{\smash{$1$}}}%
    \put(0.50018606,0.20198281){\color[rgb]{0,0,0}\makebox(0,0)[lb]{\smash{$k$}}}%
    \put(0.54123619,0.45174423){\color[rgb]{0,0,0}\makebox(0,0)[lb]{\smash{. . .}}}%
    \put(0.72444016,0.45174423){\color[rgb]{0,0,0}\makebox(0,0)[lb]{\smash{. . .}}}%
    \put(0.47480727,0.51048138){\color[rgb]{0,0,0}\makebox(0,0)[lb]{\smash{$1$}}}%
    \put(0.65292636,0.50992798){\color[rgb]{0,0,0}\makebox(0,0)[lb]{\smash{$t$}}}%
    \put(0.78870014,0.51048138){\color[rgb]{0,0,0}\makebox(0,0)[lb]{\smash{$b-1$}}}%
    \put(0.26922724,0.45174423){\color[rgb]{0,0,0}\makebox(0,0)[lb]{\smash{. . .}}}%
    \put(0.83593499,0.20389049){\color[rgb]{0,0,0}\makebox(0,0)[lb]{\smash{$n$}}}%
\end{picture}
}

\title{On $p$-almost direct products and residual properties of pure braid groups of nonorientable surfaces}

\author{Paolo Bellingeri and Sylvain Gervais}

\date{\empty}

\begin{document}

\maketitle

\begin{abstract}
We prove that the $n^{\text{th}}$ pure braid group of  
a nonorientable surface (closed or with boundary, but different from $\R\mathrm{P}^{2}$) is residually 2-finite. Consequently, this 
group is residually nilpotent. The key ingredient in the closed case is the notion of $p$-almost direct product,
which is a generalization of the notion of almost direct product. We prove  therefore also some results on lower central series
and augmentation ideals of $p$-almost direct products.
\end{abstract}

\begingroup
 \renewcommand{\thefootnote}{}
 \footnotetext{Date: 17th November 2014
 }
 \endgroup


\section{Introduction}
Let $M$ be a surface (orientable or not, possibly with boundary) and 
$F_{n}(M)=\{(x_{1},\ldots,x_{n})\in M^{n}\,/\,x_{i}\neq 
x_{j}\text{ for }i\neq j\}$ its $n^{\text{th}}$ configuration space. 
The fundamental group $\pi_{1}(F_{n}(M))$ is 
called the $n^{\text{th}}$ \emph{pure braid group} of $M$ and shall be
denoted by $P_{n}(M)$.  

The mapping class group $\Gamma(M)$ of 
$M$ is  the group of isotopy classes of homeomorphisms 
$h:M\to M$ which are identity on the boundary. Let 
$\mathcal{X}_{n}=\{z_{1},\ldots,z_{n}\}$ a set of $n$ 
distinguished points in the  interior of $M$;
the pure  mapping class group 
$\mathrm{P}\Gamma(M, \mathcal{X}_n)$ relatively to $\mathcal{X}_{n}$ is 
the  group  of the isotopy classes of  homeomorphisms $h:M\to M$ satisfying 
$h(z_{i})=z_{i}$ for all $i$: since this group does not depend on the choice of the set  $\mathcal{X}_n$ but only on its cardinality
we can write $\mathrm{P}_n\Gamma(M)$ instead of $\mathrm{P}\Gamma(M, \mathcal{X}_n)$. Forgetting the 
marked points, we get a morphism $\mathrm{P}_{n}\Gamma(M)\to\Gamma(M)$ whose kernel is known to be 
isomorphic to $P_{n}(M)$ when $M$ is not a sphere, a torus, a projective plane or a Klein  bottle (see~\cite{Sc,GJ}).

\medskip

Now, recall that if $\mathcal{P}$ is a group-theoretic property, then 
a group $G$ is said to be \emph{residually} $\mathcal{P}$ if, for all $g\in 
G$, $g\neq 1$, there exists a group homomorphism $\varphi:G\to H$ 
such that $H$ is $\mathcal{P}$ and $\varphi(g)\neq1$. We are 
interested in this paper to the following properties: nilpotence, 
being free and being a finite $p$-group for a prime number $p$ 
(mostly $p=2$). Recall 
that, if for subgroups $H$ and $K$ of $G$, $[H,K]$ is the subgroup 
generated by $\{[h,k]\,/\,(h,k)\in H\times K\}$ where 
$[h,k]=h^{-1}k^{-1}hk$, the lower central 
series of $G$, $(\Gamma_{k}G)_{k\geq 1}$,   is defined inductively by 
$\Gamma_{1}G=G$ and $\Gamma_{k+1}G=[G,\Gamma_{k}G]$. It is well known that $G$ is residually nilpotent if, 
and only if, $\build{\bigcap}{k=1}{+\infty}\Gamma_{k}G=\{1\}$. 
From  the lower central series of
$G$ one can define  another  filtration $D_1(G) \supseteq  D_2(G)
\supseteq \ldots$  setting $D_1(G)=G$, and for $i\geq 2$,
defining $D_i(G)=\{ \, x \in G \, | \,\exists n\in\N^{*}, x^n \in 
\Gamma_{i}(G)\,\}$. After Garoufalidis and Levine \cite{GLe}, this filtration is
called  \emph{rational lower central series} of $G$
and a  group $G$
is residually torsion-free nilpotent if, and only if,
$\build{\bigcap}{i=1}{\infty}D_i(G)=\{ 1\}$.

\smallskip

When $M$ is an orientable surface of positive genus (possibly with boundary) or a disc with 
holes, it is proved in~\cite{BGG} and~\cite{BB} that $P_{n}(M)$ is 
residually torsion-free nilpotent for all $n\geq 1$.  The fact that a group is residually torsion-free nilpotent has several
important consequences, notably that the group is bi-orderable~\cite{MR}
and
 residually $p$-finite~\cite{Gr}. The goal of this 
article is to study the nonorientable case and, more precisely, to 
prove the following:

\begin{thm}\label{thm:nonorientable}
   The $n^{\text{th}}$ pure braid group of a nonorientable surface 
   different from $\R\mathrm{P}^{2}$ is residually $2$-finite.
\end{thm}

In the case of $P_n(\R\mathrm{P}^{2})$ we give some partial results at the end of Section 4.
Since a finite $2$-group is nilpotent, a residually $2$-finite group 
is residually nilpotent. Thus, we have

\begin{cor}
   The $n^{\text{th}}$ pure braid group of a nonorientable surface 
   different from $\R\mathrm{P}^{2}$ is residually nilpotent. 
\end{cor}

Remark that in \cite{Go} it was shown that the $n^{\text{th}}$ pure braid group of a nonorientable surface 
is not bi-orderable and therefore  it is not residually torsion-free nilpotent.  
Let us notice also that if pure braid groups of nonorientable surfaces with boundary are residually $p$ for  a prime $p\not=2$ 
therefore pure braid groups of nonorientable  closed surfaces  are also residually $p$ (Remark \ref{rem:p}); however since
the technique proposed in the nonorientable case applies only for $p=2$,  the question if 
pure braid groups of  nonorientable surfaces  are residually $p$ for $p\not=2$ is still  open 
(recall that there are groups residually $p$ for infinitely many primes $p$  which are not 
residually torsion-free nilpotent, see \cite{H}).

\medskip

Remark that one can  prove that finite type invariants separate classical braids using the fact that 
the pure braid group  $P_n$ is residually nilpotent without torsion 
(see \cite{Pa}). Moreover using above residual properties it is possible to construct 
algebraically  a universal finite type invariant over $\Z$ on the 
classical braid group $B_n$ (see \cite{Pa}).
Similar constructions were afterwards proposed for braids on orientable surfaces
(see \cite{BF,GP}): in a further paper we will explore   the relevance of Theorem \ref{thm:nonorientable}
in the realm of finite type invariants over $\Z / 2\Z$ for braids on non orientable surfaces.

\medskip

From now on, $M=N_{g,b}$ is a nonorientable surface of genus $g$ with 
$b$ boundary components, simply denoted by $N_{g}$ when $b=0$. We will
see $N_{g}$ as a sphere $S^{2}$ with $g$ open discs removed and $g$ 
M\"{o}bius strips glued on each circle  (see 
figure~\ref{fig:generators} where  each crossed disc represents a 
M\"{o}bius strip). The surface $N_{g,b}$ is obtained from $N_{g}$ by 
removing $b$ open discs. The mapping class groups $\Gamma(N_{g,b})$ and pure 
mapping class group $P_{n}\Gamma(N_{g,b})$ will be denoted respectively 
$\Gamma_{g,b}$ and $\Gamma_{g,b}^{n}$.

\medskip

The paper is organized as follows. In Section 2, we prove  Theorem \ref{thm:nonorientable} for 
surfaces with boundary: following what the authors did in the 
orientable case (see~\cite{BGG}), we embed $P_{n}(N_{g,b})$ in a 
Torelli group. The difference  here is that  we must consider mod 2 
Torelli groups.  In Section 3 we introduce the notion of $p$-almost direct product, which generalizes  the notion of almost direct product (see Definition \ref{deft:main})
and we prove some results on lower central series and augmentations ideals of $p$-almost direct products (Theorems \ref{thm:suite exacte} and \ref{thm:augmentation})
that can be compared with similar results on almost direct products (Theorem 3.1 in \cite{FR} and Theorem 3.1 in \cite{Pa}). 

In Section 4, the existence of a split exact sequence
\[
\xymatrix{
1 \ar[r] & P_{n-1}(N_{g,1})\ar[r] & P_{n}(N_{g})\ar[r] & \pi_{1}(N_{g})\ar[r] & 1}
\]
and results from Section 2 and 3 are used to prove Theorem 1 in the closed 
case (Theorem \ref{thm:closed}). The method is similar to the one developed for 
orientable surface in~\cite{BB}: the difference will be that  the 
semi-direct product $P_{n-1}(N_{g,1})\rtimes\pi_{1}(N_{g})$   is a 
$2$-almost-direct  product (and not an almost-direct product as in 
the case of closed oriented surfaces).
 
 For proving the main result of the paper, we will also need a group presentation for 
$P_{n}(N_{g,b})$  when $b\geq 1$. Although generators  of this group
seem to be known, we could not find a group presentation  in the literature. Thus, we 
give one in the Appendix (Theorem \ref{thm:PresentationCasBord}).

\paragraph{Acknowledgments.}

The research of the first author was partially supported   by  French grant ANR-11-JS01-002-01.
The authors are grateful to Carolina de Miranda e Pereiro and John Guaschi  for useful discussions and comments.


\section{The case of non-empty boundary}\label{par:cas a bord}

In this section, $N=N_{g,b}$ is a nonorientable surface of genus 
$g\geq 1$ with boundary (\emph{ie} $b\geq 1$). In this case, one has 
$P_{n}(N)=\ker{\Gamma_{g,b}^{n}\to\Gamma_{g,b}}$ for all $n\geq 1$.

\subsection{Notations}
We will follow notations from \cite{PS}.
A simple closed curve in $N$ is an embedding $\alpha:S^{1}\to 
N\setminus\partial N$. Such a curve 
is said two-sided (resp. one-sided) if it admits a regular neighborhood homeomorphic to an annulus 
(resp. a M\"{o}bius strip). We shall consider the 
following elements in $\Gamma_{g,b}$.
\begin{list}{$\bullet$}{\leftmargin0mm\labelwidth0mm\labelsep2mm\itemindent3mm}
    \item If $\alpha$ is a two-sided simple closed curve in $N$, $\tau_{\alpha}$ is a 
    Dehn twist along $\alpha$.
    \item Let $\mu$ and $\alpha$ be two simple closed curves such that $\mu$ is 
    one-sided, $\alpha$ is two-sided and $|\alpha\cap\mu|=1$. A regular neighborhood $K$ (resp. 
    $M$) of $\alpha\cup\mu$ (resp. $\mu$) is diffeomorphic to a Klein bottle with one hole 
    (resp. a M\"{o}bius strip). Pushing $M$ once along $\alpha$, we get a diffeomorphism of $K$ 
    fixing the boundary: it can be extended via the identity to $N$. Such a 
    diffeomorphism is called a crosscap slide, and denoted by $Y_{\mu,\alpha}$.
    \item Consider a one-sided simple closed curve $\mu$ containing exactly one marked point $z_{i}$. Sliding $z_{i}$ once along $\mu$, we get a 
    diffeomorphism $S_{\mu}$ of $N$ which is identity outside a regular neighborhood of $\mu$. Such a 
    diffeomorphism will be called puncture slide along $\mu$.
\end{list}

\subsection{Blowup homomorphism}

In this section, we recall the construction of the Blowup 
homomorphism $\eta_{g,b}^{n} :\Gamma_{g,b}^{n}\to\Gamma_{g+n,b}$ 
given in~\cite{Sz1},~\cite{Sz2} and~\cite{PS}.

\medskip
Let $U=\{z\in\C\,/\,|z|\leq 1\}$ and, for $i=1,\ldots,n$, fix an 
embedding $e_{i}:U\to N$ such that $e_{i}(0)=z_{i}$, 
$e_{i}(U)\cap e_{j}(U)=\emptyset$ if $i\neq j$ and 
$e_{i}(U)\cap\partial N=\emptyset$ for all $i$. If we remove the 
interior of each $e_{i}(U)$ (thus getting the surface $N_{g,b+n}$) and identify, for each $z\in\partial U$, 
$e_{i}(z)$ with $e_{i}(-z)$, we get a nonorientable surface of genus 
$g+n$ with $b$ boundary components, that is to say a surface 
homeomorphic to $N_{g+n,b}$. Let us denote by 
$\gamma_{i}=e_{i}(S^{1})$ the boundary of $e_{i}(U)$, and by 
$\mu_{i}$ its image in $N_{g+n,b}$ : it is a one-sided simple closed 
curve.

\medskip
Now, let $h$ be an element of $\Gamma_{g,b}^{n}$. It can be 
represented by a homeomorphism $N_{g,b}\to N_{g,b}$, still 
denoted $h$, such that
\begin{list}{(\arabic{liste})}{\usecounter{liste}\leftmargin20mm}
    \item $h\big(e_{i}(z)\big)=e_{i}(z)$ if $h$ preserves local 
    orientation at $z_{i}$;
    \item $h\big(e_{i}(z)\big)=e_{i}(\bar{z})$ if $h$ reverses local
    orientation at $z_{i}$.
\end{list}

Such a homeomorphism $h$ commutes with the identification leading to 
$N_{g+n,b}$ and thus induces an element $\eta(h)\in\Gamma_{g+n,b}$. 
It is proved in~\cite{Sz2} that the map 
$\eta_{g,b}^{n}=\eta:\Gamma_{g,b}^{n}\to\Gamma_{g+n,b}$  which sends $h$ to $\eta(h)$ 
is well defined for $n=1$, but the proof also works for $n>1$. This 
homomorphism is called \emph{blowup} homomorphism.

\begin{prop}
    The blowup homomorphism 
    $\eta_{g,b}^{n}:\Gamma_{g,b}^{n}\to\Gamma_{g+n,b}$ is injective 
    if $(g+n,b)\neq(2,0)$.
\end{prop}

\begin{rem}
This result is proved in~\cite{Sz1} for $(g,b)=(0,1)$, but the proof can be adapted in our case as follows.
\end{rem}

\begin{proof}
    Suppose that $h:N_{g,b}\to N_{g,b}$ is a homeomorphism satisfying
    $h(z_{i})=z_{i}$ for all $i$ and\linebreak[4] $\eta(h):N_{g+n,b}\to N_{g+n,b}$   
    is isotopic to identity. Then $h$ is isotopic to a map equal to 
    identity on $e_{i}(U)$ for all $i$ (otherwise, $\mu_{i}$ is 
    isotopic to $\mu_{i}^{-1}$ since  $\eta(h)(\mu_{i})$ is isotopic to 
    $\mu_{i}$) and its restriction to $N_{g,b+n}$ is an element of 
    the kernel of the natural map $\Gamma_{g,b+n}\to\Gamma_{g+n,b}$ 
    induced by glueing a M\"{o}bius strip on $n$ boundary components. 
    However, this kernel is generated by the Dehn twists along the 
    curves $\gamma_{i}$ (see~\cite[theorem~3.6]{St}).
    Now, any $\gamma_{i}$ bounds a 
    disc with one marked point in $N_{g,b}$:  the corresponding  Dehn twist is trivial 
    in $\Gamma_{g,b}$ and therefore $h$ is isotopic to identity.
\end{proof}

\subsection{Embedding $\boldsymbol{P_{n}(N_{g,b})}$ in $\boldsymbol{\Gamma_{g+n+2(b-1),1}}$}

Gluing a one-holed torus onto $b-1$ boundary components of 
$N_{g,b}$, we get $N_{g,b}$ as a subsurface of $N_{g+2(b-1),1}$. This 
inclusion induces a homomorphism 
$\chi_{g,b}:\Gamma_{g,b}\to\Gamma_{g+2(b-1),1}$ which is injective 
(see~\cite{St}). Thus, the composed map 
$\lambda_{g,b}^{n}=\chi_{g+n,b}\circ\eta_{g,b}^{n}:\Gamma_{g,b}^{n}\to\Gamma_{g+n+2(b-1),1}$ is also injective.

Recall that the mod $p$ Torelli group $I_{p}(N_{g,1})$ is the subgroup of $\Gamma_{g,1}$ defined as the 
kernel of the action of $\Gamma_{g,1}$ on $H_{1}(N_{g,1};\Z/p\Z)$.
In the following we will consider in particular the case of the mod 2 Torelli group $I_{2}(N_{g,1})$.

\begin{prop}\label{prop:PureBraidInTorelli}
    If $b\geq 1$, $\lambda_{g,b}^{n}\big(P_{n}(N_{g,b})\big)$ is a 
    subgroup of the Torelli subgroup $I_{2}(N_{g+n+2(b-1),1})$.
\end{prop}

\def\svgwidth{0.6\textwidth}
\setlength{\unitlength}{\svgwidth}
\begin{figure}[h]
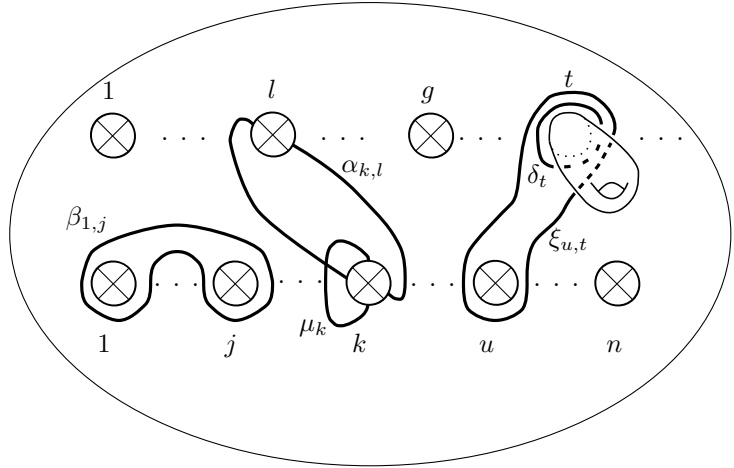

    \begin{center}
	\GenerateursTorelli
	\caption{Image of the generators of $P_{n}(N_{g,b})$ in $\Gamma_{g+n+2(b-1),1}$}
	\label{fig:generatorsdansTorelli}
    \end{center}
\end{figure}

\begin{proof}
    The image of the generators (see figures~\ref{fig:generators},~\ref{fig:generatorsBord} and 
    theorem~\ref{thm:PresentationCasBord}) $(B_{i,j})_{1\leq i<j\leq n}$, 
    $(\rho_{k,l})_{\build{}{\scriptstyle \!\!1\leq l\leq g}{\scriptstyle 1\leq k\leq n}}$ and 
    $(x_{u,t})_{\build{}{\scriptstyle 1\leq t\leq b-1}{\scriptstyle 
    \!\!\!\!\!\!1\leq u\leq n}}$ of $P_{n}(N_{g,b})$ under 
    $\lambda_{g,b}^{n}$ are respectively (see figure~\ref{fig:generatorsdansTorelli}):
    \begin{list}{$\ast$}{\leftmargin10mm}
        \item Dehn twist along curves $\beta_{i,j}$, which bound a subsurface 
	homeomorphic to $N_{2,1}$;
	\item crosscap slides $Y_{\mu_{k},\alpha_{k,l}}$;
	\item the product $\tau_{\xi_{u,t}}\tau_{\delta_{t}}^{-1}$ of 
	Dehn twists along the bounding curves $\xi_{u,t}$ and 
	$\delta_{t}$.
    \end{list}
    According to \cite{Sz2}, all these elements are in 
    the mod 2 Torelli subgroup $I_{2}(N_{g+n+2(b-1),1})$.
    \end{proof}

\begin{rem}
The embedding provided in Proposition~\ref{prop:PureBraidInTorelli} does not hold
for  $I_{p}(N_{g+n+2(b-1),1})$, when $p\not=2$: for exemple, the cross slide 
$Y_{\mu_{k},\alpha_{k,l}}$ is not in the mod $p$ Torelli subgroup since it sends $\mu_{k}$ to 
$\mu_{k}^{-1}$.
\end{rem}

\subsection{Conclusion of the proof}

We shall use the following result, which is a straightforward 
consequence of a similar result for mod $p$ Torelli groups of
orientable surfaces  due to L. Paris \cite{P}:

\begin{thm} \label{thm:P}
   Let $g \ge 1$. The mod p Torelli group $I_{p}(N_{g,1})$ is residually p-finite.
\end{thm}
\begin{proof}
We use Dehn-Nielsen-Baer Theorem (see for instance Theorem  5.15.3 of \cite{CVZ}) which states that $\Gamma_{g,1}$ embeds in $Aut(\pi_1(N_{g,1}))$.
Since $\pi_1(N_{g,1})$ is free we can apply Theorem  1.4 in \cite{P} 
which claims  that if $G$ is a free group, its  mod $p$ Torelli group
(\emph{i.e.}  the kernel of the canonical map from $Aut(G)$ to 
$GL(H_1(G,\FF_p)$) is residually $p$-finite.
Therefore $I_{p}(N_{g,1})$ is residually $p$-finite.
\end{proof}

\begin{thm}\label{thm:casbord}
 Let $g\ge 1$, $b>0$, $n\ge 1$.  $P_{n}(N_{g,b})$ is residually $2$-finite.
\end{thm}
\begin{proof}
The group  $P_{n}(N_{g,b})$ is a subgroup of 
$I_{2}(N_{g+n+2(b-1),1})$ by 
Proposition~\ref{prop:PureBraidInTorelli} and by injectivity of the map $\lambda_{g,b}^{n}$. Then by Theorem  \ref{thm:P} it follows that 
$P_{n}(N_{g,b})$ is residually $2$-finite.
\end{proof}


\section{$\boldsymbol{p}$-almost direct products}

\subsection{On residually $\boldsymbol{p}$-finite groups}

Let $p$ be a prime number and $G$ a group. If $H$ is a subgroup of $G$, we denote by $H^{p}$ the 
subgroup generated by $\{h^{p}\,/\,h\in H\}$. Following ~\cite{P}, we define the lower $\FF_{p}$-linear central 
filtration $(\gamma_{n}^{p}G)_{n\in\N^{*}}$ of $G$ by $\gamma_{1}^{p}G=G$ and, for 
$n\geq 1$, $\gamma_{n+1}^{p}G$ is the subgroup of $G$ generated by $[G,\gamma_{n}^{p}G]\cup 
(\gamma_{n}^{p}G)^{p}$. Note that the subgroups $\gamma_{n}^{p}G$ are characteristic in $G$ and that 
the quotient group $G/\gamma_{2}^{p}G$ is nothing but the first homology group 
$H_{1}(G;\FF_{p})$. The followings are proved in~\cite{P}:
\begin{list}{$\bullet$}{\leftmargin10mm}
    \item for $m,n\geq 1$, $[\gam[m]{G},\gam{G}]\subset\gam[m+n]{G}$;
    \item a finitely generated group $G$ is a finite $p$-group if, and only if, there exists some 
    $N\geq 1$ such that $\gam[N]{G}=\{1\}$;
    \item a finitely generated group $G$ is residually $p$-finite if, and only if, 
    $\build{\cap}{n=1}{+\infty}\gam{G}=\{1\}$;
\end{list}
and clearly, if $f:G\to G'$ is a group homomorphism, then $f(\gam{G})\subset\gam{G'}$ for all $n\geq 
1$.

\begin{deft}\label{deft:main}
 Let $\xymatrix{1 \ar[r] & A\ar@{^{(}->}[r] & B\ar[r]^{\lambda} & C\ar[r]\ar@/^1pc/@{.>}[l]^{\sigma} & 1}$ 
 be a split exact sequence.
\begin{itemize}
\item If the action of $C$ induced on $H_{1}(A;\Z)$ is 
    trivial (i.e. the action is trivial on $A^{Ab}=A/[A,A]$), we say that  B is an almost direct  product of $A$ and $C$.
  \item   If the action of $C$ induced on $H_{1}(A;\FF_{p})$ is 
    trivial (i.e. the action is trivial on $A/\gamma_{2}^{p}A$), we say that $B$ is a $p$-almost direct product of $A$ and $C$.
    \end{itemize}
\end{deft}

Let us remark that, as in the case of almost direct products (Proposition 6.3 of \cite{BGoGu}), the fact to be a $p$-almost direct product does not depend on the choice of the section.

\begin{prop} \label{prop:allpalmost}
Let $\xymatrix{
1 \ar[r] & A\ar@{^{(}->}[r] & 
B\ar[r]^{\lambda} & 
C\ar[r] & 1}$ be a split exact sequence of
groups. Let $\sigma,\sigma'$ be sections for $\lambda$, and suppose that the induced
action of $C$ on $A$ via $\sigma$ on $H_{1}(A;\FF_{p})$ is trivial. Then the same is true for the
section $\sigma'$.
\end{prop}

\begin{proof}
Let $a\in A$ and $c\in C$. By hypothesis, $\sigma(c)\, a \,
(\sigma(c))^{-1}\equiv a \bmod {\gamma_{2}^{p}A}$. Let $\sigma'$ be another section for
$\lambda$. Then $\lambda \circ \sigma'(c)=\lambda \circ \sigma(c)$, and so $\sigma'(c) \, (\sigma(c))^{-1}\in
\ker \lambda$. Thus there exists $a'\in A$ such that $\sigma'(c)=a' \, \sigma(c)$, and
hence
\begin{equation*}
\sigma'(c)\, a \, (\sigma'(c))^{-1}\equiv a'\, \sigma(c)\,a\, (\sigma(c))^{-1} \,
a'^{-1}\equiv a'aa'^{-1}\equiv a \bmod{\gamma_{2}^{p}A}.
\end{equation*}
Thus the induced action of $C$ on $H_{1}(A;\FF_{p})$ via $\sigma'$ is also trivial.
\end{proof}

\medskip
The first  goal of this section is to prove the following Theorem (see Theorem 3.1 in \cite{FR}
for an analogous result for almost direct products).

\begin{thm}\label{thm:suite exacte}
    Let $\xymatrix{
1 \ar[r] & A\ar@{^{(}->}[r] & 
B\ar[r]^{\lambda} & 
C\ar[r]\ar@/^1pc/@{.>}[l]^{\sigma} & 1}$ be a split exact sequence where  $B$ is a $p$-almost direct product of $A$ and $C$. 
Then, for all $n\geq 1$, one has a split exact sequence
    $$
    \xymatrix{
1 \ar[r] & \gam{A}\ar@{^{(}->}[r] & 
\gam{B}\ar[r]^{\lambda_{n}} & 
\gam{C}\ar[r]\ar@/^1pc/@{.>}[l]^{\sigma_{n}} & 1}
    $$
    where $\lambda_{n}$ and $\sigma_{n}$ are restrictions of $\lambda$ and $\sigma$.
\end{thm}

We shall need the following preliminary result.

\begin{lem}\label{lem:lemme FR}
    Under the hypotheses of Theorem~\ref{thm:suite exacte}, one has, for all $m,n\geq 1$ 
    $$[\gam[m]{C'},\gam{A}]\subset\gam[m+n]{A}$$
    where $C'$ denotes $\sigma(C)$.
\end{lem}

\begin{proof}
    First, we prove by induction on $n$ that $[C',\gam{A}]\subset\gam[n+1]{A}$ for all $n\geq 1$. 
    The cas $n=1$ corresponds to the hypotheses: the action of $C$ on $H_{1} (A;\FF_{p})=A/\gam[2]{A}$ 
    is trivial if, and only if, $[C',A]\subset\gam[2]{A}$. Thus, suppose that 
    $[C',\gam{A}]\subset\gam[n+1]{A}$ for some $n\geq 1$ and let us prove that 
    $[C',\gam[n+1]{A}]\subset\gam[n+2]{A}$. In view of the definition of $\gam[n+1]{A}$, we have to 
    prove that $\big[C',[A,\gam{A}]\big]\subset\gam[n+2]{A}$ and 
    $[C',(\gam{A})^{p}]\subset\gam[n+2]{A}$. For the first case, we use a classical result 
    (see~\cite{MKS}, theorem 5.2) which says
    \[
    \big[C',[A,\gam{A}]\big]=\big[\gam{A},[C',A]\big]\big[A,[\gam{A},C']\big].
    \]
    We have just seen that $[C',A]\subset\gam[2]{A}$ thus 
    $\big[\gam{A},[C',A]\big]\subset[\gam{A},\gam[2]{A}]\subset\gam[n+2]{A}$. Then, the induction 
    hypotheses says that $[\gam{A},C']\subset\gam[n+1]{A}$ thus 
    $\big[A,[\gam{A},C']\big]\subset[A,\gam[n+1]{A}]\subset\gam[n+2]{A}$. The second case works as 
    follows: for $c\in C'$ and $x\in\gam{A}$, one has, using the fact that 
    $[u,vw]=[u,w][u,v]\big[[u,v],w\big]$ (see~\cite{MKS})
    \[
    [c,x^{p}]=[c,x][c,x^{p-1}]\big[[c,x^{p-1}],x\big]=
    \cdots=[c,x]^{p}\big[[c,x],x\big]\big[[c,x^{2}],x\big]\cdots\big[[c,x^{p-1}],x\big].
    \]
    Since $c\in C'$ and $x\in\gam{A}$, one has 
    $[c,x^{i}]\in[C',\gam{A}]\subset\gam[n+1]{A}$ for all $i$, $1\leq 
    i\leq p-1$, which leads 
    to $[c,x]^{p}\in(\gam[n+1]{A})^{p}\subset\gam[n+2]{A}$ and 
    $\big[[c,x^{i}],x\big]\in[\gam[n+1]{A},A]\subset\gam[n+2]{A}$.
    
    \smallskip
    Now, we suppose that $[\gam[m]{C'},\gam{A}]\subset\gam[m+n]{A}$ for some $m\geq 1$ and all 
    $n\geq 1$ and prove that 
    $[\gam[m+1]{C'},\gam{A}]\subset\gam[m+n+1]{A}$. As above, there 
    are two cases which work on the same way:
    
    \medskip
    \textit{(i)}
    $\begin{array}[t]{rcl}
    \big[[C',\gam[m]{C'}],\gam{A}\big]&=&\big[[\gam{A},C'],\gam[m]{C'}\big]\big[[\gam[m]{C'},\gam{A}],C'\big]\\
    &\subset&\big[\gam[n+1]{A},\gam[m]{C'}\big]\big[\gam[m+n]{A},C'\big]\\
    &\subset&\gam[m+n+1]{A}.
    \end{array}$
    
    \smallskip
    \textit{(ii)} For $c\in\gam[m]{C'}$ and $x\in\gam{A}$, one has
    \[
    [c^{p},x]=\big[c,[x,c^{p-1}]\big][c^{p-1},x][c,x]=\cdots=\big[c,[x,c^{p-1}]\big]\cdots\big[c,[x,c]\big][c,x]^{p}
    \]
    which is an element of $\gam[m+n+1]{A}$ by induction hypotheses.
\end{proof}

\begin{prooftext}{Theorem~\ref{thm:suite exacte}}
    Restrictions of $\lambda$ and $\sigma$ give rise to morphisms 
    $\lambda_{n}:\gam{B}\to\gam{C}$ and 
    $\sigma_{n}:\nolinebreak[4]\gam{C}\to\nolinebreak[4]\gam{B}$ satisfying 
    $\lambda_{n}\circ\sigma_{n}=\id_{\gam{C}}$, $\sigma_{n}$ is 
    onto and $\ker{\lambda_{n}}=A\cap\gam{B}$. Thus, we need to prove 
    that $A\cap\gam{B}=\gam{A}$ for all $n\geq 1$. Clearly one has 
    $\gam{A}\subset A\cap\gam{B}$. In order to prove the converse 
    inclusion, we follow the method developed in~\cite{FR} for 
    almost semi-direct product and define $\tau:B\to B$ by 
    $\tau(b)=\big(\sigma\lambda(b)\big)^{-1}b$. This map has the 
    following properties:
    \begin{list}{{\it 
	(\roman{liste})}}{\leftmargin10mm\usecounter{liste}}
        \item since $\lambda\sigma=\id_{C}$, $\tau(B)\subset A$;
	\item for $x\in B$, $\tau(x)=x$ if, and only if, $x\in A$;
	 \item for $(b_{1},b_{2})\in B^{2}$, 
	 $\tau(b_{1}b_{2})=[\sigma\lambda(b_{2}),\tau(b_{1})^{-1}]\tau(b_{1})\tau(b_{2})$;
	 \item for $b\in B$, setting $a=\tau(b)$ and 
	 $c=\sigma\lambda(b)$, we get $b=ca$ with $c\in C'=\sigma(C)$ 
	 and $a\in A$, this decomposition being unique.
    \end{list}
    We claim that $\tau(\gam{B})\subset\gam{A}$ for all $n\geq 1$. 
    From this, we conclude easily the proof: if $x\in A\cap\gam{B}$, 
    then $x=\tau(x)\in\gam{A}$.
    
    One has $\tau(\gam[1]{B})\subset\gam[1]{A}$. Suppose 
    inductively that $\tau(\gam{B})\subset\gam{A}$ for some $n\geq 1$ and let us prove that 
    $\tau(\gam[n+1]{B})\subset\gam[n+1]{A}$. Suppose first that $x$ is an element of $\gam{B}$. 
    Then using \textit{(iii)} we get:
    \begin{eqnarray*}
    \tau(x^{p}) & = & [\sigma\lambda(x),\tau(x^{p-1})^{-1}]\tau(x^{p-1})\tau(x)\\
    &\vdots&\\
    &=&[\sigma\lambda(x),\tau(x^{p-1})^{-1}][\sigma\lambda(x),\tau(x^{p-2})^{-1}]\cdots[\sigma\lambda(x),\tau(x)^{-1}]\tau(x)^{p}.
    \end{eqnarray*}
    Since $\sigma\lambda(x)\in\gam{C'}$ and, by induction hypotheses, 
    $\tau(x^{i})\in\gam{A}$ for $1\leq i\leq p-1$, we get\linebreak[4] 
    $\tau(x^{p})\in[\gam{C'},\gam{A}]\cdot(\gam{A})^{p}\subset\gam[n+1]{A}$ by lemma~\ref{lem:lemme FR} : this 
    prove that $\tau\big((\gam{B})^{p}\big)\subset\gam[n+1]{A}$. Next, let $b\in B$ and 
    $x\in\gam{B}$. Setting $a=\tau(b)\in A$, $y=\tau(x)\in\gam{A}$ by induction hypotheses, 
    $c=\sigma\lambda(b)\in C'$ and $z=\sigma\lambda(x)\in\gam{C'}$, we get
    \begin{eqnarray*}
    \tau\big([b,x]\big) & = & 
    \Big(\sigma\lambda\big([b,x]\big)\Big)^{-1}[b,x]\\
    &=&\big[\sigma\lambda(b),\sigma\lambda(x)\big]^{-1}[b,x]\\
    &=& [c,z]^{-1}[ca,zy]=[z,c]a^{-1}c^{-1}y^{-1}z^{-1}cazy\\
    &=&[z,c]\big(a^{-1}c^{-1}y^{-1}cya\big)\big(a^{-1}y^{-1}c^{-1}z^{-1}czya\big)
    \big(a^{-1}y^{-1}z^{-1}azy\big)\\
    &=&[z,c]\big(a^{-1}[c,y]a\big)\big(a^{-1}y^{-1}[c,z]ya\big)\big(a^{-1}y^{-1}ay\big)\big(y^{-1}a^{-1}z^{-1}azy\big)\\
    &=&[z,c]\big(a^{-1}[c,y]a\big)\big(a^{-1}y^{-1}[c,z]ya\big)[a,y]\big(y^{-1}[a,z]y\big)\\
    &=&\Big[[c,z],\big(a^{-1}[y,c]a\big)\Big]\big(a^{-1}[c,y]a\big)[z,c]\big(a^{-1}y^{-1}[c,z]ya\big)[a,y]\big(y^{-1}[a,z]y\big)\\
    &=&\Big[[c,z],\big(a^{-1}[y,c]a\big)\Big]\big(a^{-1}[c,y]a\big)\Big[[c,z],ya\Big][a,y]\big(y^{-1}[a,z]y\big).
    \end{eqnarray*}
    Now, $[c,z]\in[C',\gam{C'}]\subset\gam[n+1]{C'}$, $[y,c]\in[\gam{A},C']\subset\gam[n+1]{A}$ 
    (lemma~\ref{lem:lemme FR}) thus $\big[[c,z],\big(a^{-1}[y,c]a\big)\big]\in\gam[n+1]{A}$. Then, 
    $\big[[c,z],ya\big]\in[\gam[n+1]{C'},A]\subset\gam[n+1]{A}$, 
    $[a,y]\in[A,\gam{A}]\subset\gam[n+1]{A}$ and $[a,z]\in[A,\gam{C'}]\subset\gam[n+1]{A}$. Thus, 
    $\tau([b,x])\in\gam[n+1]{A}$ and $\tau([B,\gam{B}])\subset\gam[n+1]{A}$.
\end{prooftext}

\begin{cor}\label{cor:residually p-finite}
        Let $\xymatrix{
1 \ar[r] & A\ar@{^{(}->}[r] & 
B\ar[r]^{\lambda} & 
C\ar[r]\ar@/^1pc/@{.>}[l]^{\sigma} & 1}$ be a split exact sequence such that $B$ is a $p$-almost direct product of $A$ and $C$. If $A$ and $C$ are residually $p$-finite, then $B$ is residually $p$-finite.
\end{cor}

\subsection{Augmentation ideals}

Given a group $G$ and $\K=\Z$ or $\FF_2$ we will denote by
$\K[G]$ the  group ring of $G$ over $\K$ and by
$\overline{\K [G]}$  the augmentation ideal of $G$.
The group ring $\K[G]$ is filtered by the powers $\overline{\K [G]}^j$ of $\overline{\K [G]}$
and we can define the associated graded algebra $gr(\K[G])=\oplus \overline{\K [G]}^j/\overline{\K [G]}^{j+1}$.

The following theorem provides a decomposition formula for the augmentation ideal of a $2-$almost direct product 
(see Theorem 3.1 in \cite{Pa} for an analogous in the case  of almost 
direct products).\\
Let $A\rtimes C$ be a semi-direct product between two groups $A$ and 
$C$. It is a classical result that the map $a\otimes c \mapsto ac$ 
induces a $\K$-isomorphism from $\K[A]\otimes\K[C]$ to $\K[A\rtimes 
C]$. Identifying these two $\K$-modules, we have the following:

\begin{thm} \label{thm:augmentation}
If $A\rtimes C$ is  a $2-$almost direct product then :
$$ \overline{\FF_2 [A \rtimes C]}^k= \displaystyle \sum_{i+h=k}   \overline{\FF_2 [A]}^i \otimes \overline{\FF_2 [C]}^h \qquad \mbox{for all} \;  k.$$
\end{thm}
\begin{proof}
We sketch the proof which is almost verbatim the same as the proof of Theorem 3.1 in \cite{Pa}.
Let  $R_k=\displaystyle \sum_{i+h=k}   \overline{\FF_2 [A]}^i \otimes 
\overline{\FF_2 [C]}^h$; $R_k$ is a descending filtration on $\FF_2 
[A] \otimes \FF_2 [C]$, and with the above identification, we get that $R_k \subset  
\overline{\FF_2 [A \rtimes C]}^k$. To verify the other inclusion we 
have to check that $\displaystyle \prod_{j=1}^k (a_j c_j -1) \in R_k$ 
for every $a_1,\ldots, a_k$ in $A$ and $c_1,\ldots,c_k$ in $C$. 
Actually it is enough to verify that $e=\displaystyle \prod_{j=1}^k 
(e_j -1) \in R_k$ either $e_j \in A$ or $e_j \in C$  (see Theorem 3.1 
in \cite{Pa} for a proof of this fact):  we call $e$ a \emph{special} 
element. We associate to a special element $e$ an element in 
$\{0,1\}^k$: let $type(e)=(\delta(e_1), \ldots, \delta(e_k))$ where 
$\delta(e_j)=0$ if $e_j \in A$ and $\delta(e_j)=1$ if $e_j \in C$. We 
will say that the special element $e$ is \emph{standard} if
 
 $$\mbox{type}(e)=(\overbrace{0, \ldots,0}^i,\overbrace{1, \ldots,1}^h)$$
 
 In this case $e \in  \overline{\FF_{2} [A]}^i \otimes 
 \overline{\FF_{2} [C]}^h \subset R_k$ and we are done.
 We claim that we can reduce all special elements to linear combinations of standard elements.
 If $e$ is not standard, then it must be of the form
 
 $$e= \displaystyle \prod_{i=1}^r(a_i-1) \prod_{j=1}^s(c_i-1) (c-1)(a-1) \prod_{l=1}^t(e_i-1)$$
 where $a_1,\ldots, a_r,  a \in A$, $c_1,\ldots, c_s,  c \in A$, $\tilde{e}=\prod_{l=1}^t(e_i-1)$ is special and
 $r+s+t+2=k$.  Therefore 
  $$\mbox{type}(e)=(\overbrace{0, \ldots,0}^r,\overbrace{1, 
  \ldots,1}^s,1,0,\delta(e_1), \ldots,\delta(e_t))\, .$$

Now we can use the assumption that $A\rtimes C$ is a $2-$almost direct product to claim that one has commutation relations in 
$\Z[A\rtimes C] $ expressing the difference  $(c-1)(a-1)-  (a-1)(c-1)$ as a linear combination of terms of the form
$$(a'-1)(a''-1)c \qquad \mbox{with}\; a', a'' \in A $$
for any $a\in A$ and $c\in C$. In fact,
$$
(c-1)(a-1)-  (a-1)(c-1)=ca-ac=(cac^{-1}a^{-1}-1)ac=(f-1)ac
$$
where $f=[c^{-1},a^{-1}]\in[C,A]\subset \gamma_2^2(A)$ by lemma~\ref{lem:lemme FR}.  We can decompose $f$ as $f=h_1 k_1 \cdots h_m k_m$
where,  for $j=1,\ldots, m$,  $h_j$ belongs to $[A, A]$ and $k_j = (k'_j)^2$ for some $k'_j \in A$.
One knows (see for instance \cite{Ch} p. 194) that   for $j=1,\ldots, m$  $(h_j -1)$ is a linear combination of terms of the form
$$
(h'_j-1)(h''_j-1)\alpha_{j}  \qquad \mbox{with}\; h'_j, 
h''_j,\alpha_{j}\in A.
$$
On the other hand  for $j=1,\ldots, m$  we have also that 
$$
(k_j-1)= (k'_j-1)(k'_j-1) \qquad \mbox{with}\; k'_j\in A 
\quad\mbox{since the coefficients are }\FF_{2}.
$$
Then, recalling that $(hk-1)=(h-1)k+(k-1)$ for any $h,k\in A$, we can conclude 
that  $f-1$ can be rewritten as 
 a linear combination of terms of the form
$$
(f'-1)(f''-1)\alpha  \qquad \mbox{with}\; f',f'',\alpha\in A \label{3.1}
$$
and that  $(c-1)(a-1)-  (a-1)(c-1)$
is a linear combination of terms of the form 
$$
(f'-1)(f''-1)\alpha c  \qquad \mbox{with}\; f',f'',\alpha\in A.
$$
Rewriting $(f''-1)\alpha$ as $(f''\alpha-1)-(\alpha-1)$
we obtain  that the difference  $(c-1)(a-1)-  (a-1)(c-1)$ can be seen as a linear combination of terms of the form
$$(a'-1)(a''-1)c \qquad \mbox{with}\; a', a'' \in A. $$

Therefore $e$ can be rewritten as a sum whose first term is the special element   
$$e'=\displaystyle \prod_{i=1}^r(a_i-1) \prod_{j=1}^s(c_i-1) (a-1)(c-1) \prod_{l=1}^t(e_i-1)$$
and whose second term is a linear  combination of elements of the form $e'' c$ where
$$e''=\displaystyle \prod_{i=1}^r(a_i-1) \prod_{j=1}^s(c_i-1) 
(a'-1)(a''-1) \prod_{l=1}^t(c e_ic^{-1}-1)c.$$
Using the lexicographic order from the left, one has type$(e) > $type$(e')$ and type$(e) > $type$(e'')$.

By induction on the lexicographic order we  infer that $e'$ and $e''$ belong to $R_k$: since
$R_k \cdot c \subset R_k$ for any $c \in C$ it follows that $e$  belongs to $R_k$ and we are done.
\end{proof}

\section{The closed case}

\subsection{A presentation of $\boldsymbol{P_{n}(N_{g})}$ and induced identities}

We recall a group presentation of $P_{n}(N_{g})$ given in~\cite{GG3}:
the geometric interpretation of generators is provided in Figure \ref{fig:generators}.

\settowidth{\longueur}{{\underline{\bf relations: }}}
\settowidth{\longueurA}{$B_{s,j}^{-1}B_{r,j}^{-1} B_{s,j} B_{r,j} B_{i,j} B_{r,j}^{-1} 
    B_{s,j}^{-1} B_{r,j} B_{s,j}$}
\settowidth{\longueurB}{$\text{if }i<r<s<j\text{ or }r<s<i<j$}

\begin{thm}[\cite{GG3}]\label{thm:Presentation cas ferme}
    For $g\geq 2$ and $n\geq 1$, $P_{n}(N_{g})$ has the following 
    presentation:
    
    \medskip
    \noindent\ \underline{\bf generators:} $(B_{i,j})_{1\leq i<j\leq n}$ 
    and $(\rho_{k,l})_{\build{}{\scriptstyle \!\!1\leq l\leq 
    g}{\scriptstyle 1\leq k\leq n}}$.
    
    \medskip
    \noindent$\begin{array}[t]{lcl}
    \multicolumn{3}{l}{\underline{\bf relations:}\textnormal{ 
    (a)}\text{ for all }1\leq 
    i<j\leq n\textnormal{ and }1\leq r<s\leq n,}\\
    B_{r,s}B_{i,j}B_{r,s}^{-1}&=&
    \left\{\begin{array}{lll}
    B_{i,j} & \text{if }i<r<s<j\text{ or }r<s<i<j & (\textnormal{a}_{1})\\
    B_{i,j}^{-1} B_{r,j}^{-1}  B_{i,j} B_{r,j} B_{i,j} & \text{if 
    $r<i=s<j$}& (\textnormal{a}_{2})\\
    B_{s,j}^{-1} B_{i,j} B_{s,j} & \text{if $i=r<s<j$}& (\textnormal{a}_{3})\\
    B_{s,j}^{-1}B_{r,j}^{-1} B_{s,j} B_{r,j} B_{i,j} B_{r,j}^{-1} 
    B_{s,j}^{-1} B_{r,j} B_{s,j}   & \text{if $r<i<s<j$}& (\textnormal{a}_{4})
    \end{array}\right.\\
    \multicolumn{3}{l}{}\\
    \multicolumn{3}{l}{\hspace*{\longueur}
    \textnormal{ (b)}\text{ for all 
    }1\leq i<j\leq n\text{ and }1\leq k,l\leq g,}\\
    \rho_{i,k}\rho_{j,l}\rho_{i,k}^{-1}&=&
    \left\{\begin{array}{p{\longueurA}p{\longueurB}l}
    $\rho_{j,l}$ & $\text{if }k<l$ & (\textnormal{b}_{1})\\
    $\rho_{j,k}^{-1} B_{i,j}^{-1}  \rho_{j,k}^2$ & $\text{if }k=l$ & 
    (\textnormal{b}_{2})\\
    $\rho_{j,k}^{-1} B_{i,j}^{-1}\rho_{j,k} B_{i,j}^{-1} \rho_{j,l}
    B_{i,j} \rho_{j,k}^{-1} B_{i,j} \rho_{j,k}$ & $\text{if }k>l$ & 
    (\textnormal{b}_{3})
    \end{array}\right.\\
    \multicolumn{3}{l}{}\\
    \multicolumn{3}{l}{\hspace*{\longueur} \textnormal{ (c)}
    \text{ for all }1\leq i\leq n,\ 
    \rho_{i,1}^{2}\cdots\rho_{i,g}^{2}=T_{i}\text{ where 
    }T_{i}=B_{1,i}\cdots B_{i-1,i}B_{i,i+1}\cdots 
    B_{i,n}\qquad\quad(\textnormal{c})}\\
    \multicolumn{3}{l}{}\\
    \multicolumn{3}{l}{\hspace*{\longueur}
    \textnormal{ (d)}\text{ for all 
    }1\leq i<j\leq n,\ 1\leq k\leq n,\ k\neq j\text{ and }1\leq l\leq g,}\\
    \rho_{k,l}B_{i,j}\rho_{k,l}^{-1}&=&
    \left\{\begin{array}{p{\longueurA}p{\longueurB}l}
    $B_{i,j}$ & $\text{if }k<i\text{ or }j<k$ & (\textnormal{d}_{1})\\
    $\rho_{j,l}^{-1} B_{i,j}^{-1} \rho_{j,l}$ & $\text{if }k=i$ & 
    (\textnormal{d}_{2})\\
    $\rho_{j,l}^{-1} B_{k,j}^{-1} \rho_{j,l} B_{k,j}^{-1} B_{i,j} 
    B_{k,j} \rho_{j,l}^{-1} B_{k,j} \rho_{j,l}$ & $\text{if }i<k<j$ & 
    (\textnormal{d}_{3})
    \end{array}\right.\\
\end{array}$
\end{thm}

\def\svgwidth{0.5\textwidth}
\setlength{\unitlength}{\svgwidth}
\begin{figure}[h]
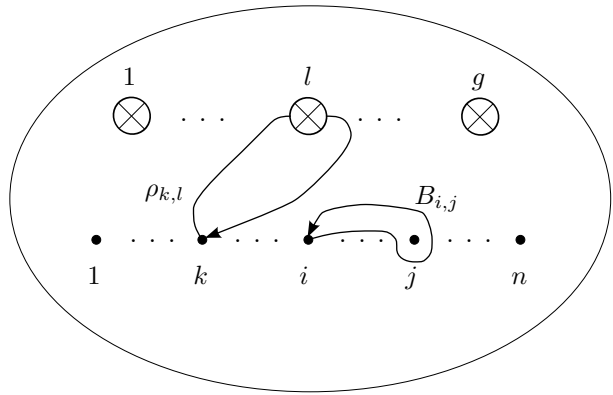

    \begin{center}
	\Generateurs
	\caption{Generators of $P_{n}(N_{g})$}
	\label{fig:generators}
    \end{center}
\end{figure}

\bigskip
For $1\leq k\leq g$, let us consider the element $a_{k}$ in $P_{n}(N_{g})$ given by 
$a_{k}=\rho_{k,g-1}\rho_{k,g}$ and set $U=a_{n}\cdots a_{2}$.

\begin{lem}\label{lem:relations}
    The following relations holds in $P_{n}(N_{g})$:
    
    \medskip\renewcommand{\arraystretch}{1.4}\setcounter{liste}{0}\setcounter{listes}{4}
    $\begin{array}{cl@{\hspace*{1cm}}l}
        \textit{(\svte)} & [\rho_{i,k},\rho_{j,k}^{-1}]=B_{i,j}^{-1}\text{ for }1\leq 
	i<j\leq n\text{ and }1\leq k\leq g; & (\textnormal{\eqsvte})\\
		
	\textit{(\svte)} & U \text{ commutes with }\rho_{1,l} \text{ for } 1\leq l\leq g-2;& 
	(\textnormal{\eqsvte}_{1})\\

	\textit{(\svte)} & [\rho_{1,g-1},U^{-1}]=T_{1}^{-1};& 
	(\textnormal{\eqprec}_{2})\\
	
	\textit{(\svte)} & a_{k}a_{j}a_{k}^{-1} \text{ commutes with }B_{i,k} \text{ for }1\leq 
	i<j<k\leq n;& (\textnormal{\eqsvte})\\
	
	\textit{(\svte)} & a_{n}a_{n-1}\cdots a_{1}\text{  commutes with }B_{j,k} \text{ for } 1\leq 
	j<k\leq n;& (\textnormal{\eqsvte})\\
	
	\textit{(\svte)} & U \text{ commutes with }B_{i,j} \text{ for 
	} 2\leq i<j\leq n;& (\textnormal{\eqsvte})\\
	
	\textit{(\svte)} & a_{n}a_{n-1}\cdots a_{1} \text{ commutes 
	with }T_{1};& (\textnormal{\eqsvte})\\
	
	\textit{(\svte)} & T_{1} \text{ commutes with }B_{j,k} \text{ 
	for } 2\leq j<k\leq n;& (\textnormal{\eqsvte})\\
    \end{array}$
\end{lem}

\begin{proof}
    Some of these identities can easily be verified drawing 
    corresponding braids. This is the case for example for the first,
    the fourth and the eighth ones (see figure~\ref{fig:identityE},~\ref{fig:identityG}
    and~\ref{fig:identityK}). Let us give an algebraic proof for the 
    others.
    
    \def\svgwidth{0.7\textwidth}
    \setlength{\unitlength}{\svgwidth}
    \begin{figure}[h]
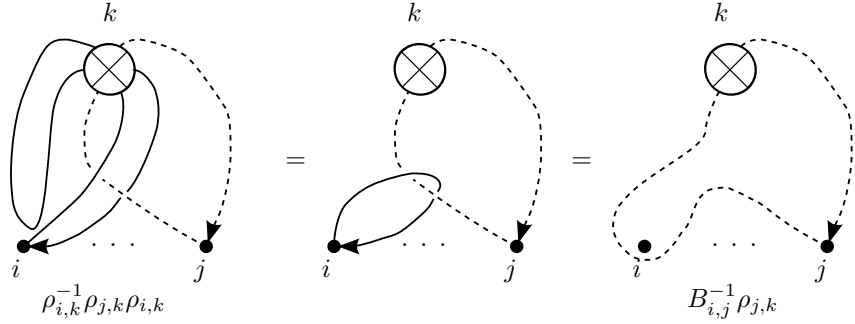

	\begin{center}
	    \identityE
	    \caption{identity $(e)$}
	    \label{fig:identityE}
	\end{center}
    \end{figure}
    
    \def\svgwidth{0.6\textwidth}
    \setlength{\unitlength}{\svgwidth}
    \begin{figure}[h]
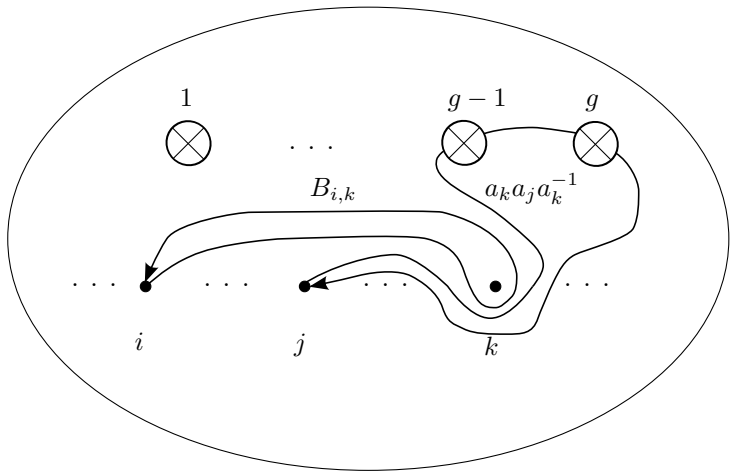

	\begin{center}
	    \identityG
	    \caption{identity $(g)$}
	    \label{fig:identityG}
	\end{center}
    \end{figure}

    \def\svgwidth{0.4\textwidth}
    \setlength{\unitlength}{\svgwidth}
    \begin{figure}[h]
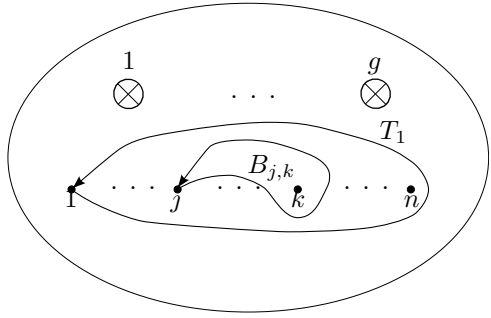

	\begin{center}
	    \identityK
	    \caption{identity $(k)$}
	    \label{fig:identityK}
	\end{center}
    \end{figure}

    \medskip
    \textit{(2)} This is a direct consequence of the definitions of 
    the $a_{i}$'s and $U$, and relations $(\textnormal{b}_{1})$.

    \medskip
    \textit{(3)} By relation ($\textnormal{b}_{1}$), $\rho_{1,g-1}$ commutes with
    $\rho_{j,g}$ for $2\leq j\leq n$. Thus, using relation 
    ($\textnormal{e}$), we 
    get:
    \begin{eqnarray*}
	\rho_{1,g-1}^{-1}U\rho_{1,g-1}&=&\rho_{1,g-1}^{-1}a_{n}\cdots 
	a_{2}\rho_{1,g-1}\\
	&=&\rho_{1,g-1}^{-1}(\rho_{n,g-1}\rho_{n,g})\cdots(\rho_{2,g-1}\rho_{2,g})\rho_{1,g-1}\\
	&=&(B_{1,n}^{-1}\rho_{n,g-1}\rho_{n,g})\cdots(B_{1,2}^{-1}\rho_{2,g-1}\rho_{2,g})\\
	&=& B_{1,n}^{-1}B_{1,n-1}^{-1}\cdots 
	B_{1,2}^{-1}(\rho_{n,g-1}\rho_{n,g})\cdots(\rho_{2,g-1}\rho_{2,g})\quad\text{by (}\textnormal{d}_{1})\\
	&=& T_{1}^{-1}U.
    \end{eqnarray*}
    
    \medskip
    \textit{(5)} Let $j$ and $k$ be integers such that $1\leq j<k\leq 
    n$. By $(\textnormal{d}_{1})$, 
    $a_{1},\ldots,a_{j-1}$ commute with $B_{j,k}$. Then, one has
    \begin{eqnarray*}
	a_{j}B_{j,k}a_{j}^{-1} &=& 
	\rho_{j,g-1}\rho_{j,g}B_{j,k}\rho_{j,g}^{-1}\rho_{j,g-1}^{-1}\\
	&=& 
	\rho_{j,g-1}\rho_{k,g}^{-1}B_{j,k}^{-1}\rho_{k,g}\rho_{j,g-1}^{-1}\quad \text{by }(\textnormal{d}_{2})\\
	&=& 
	\rho_{k,g}^{-1}\rho_{j,g-1}B_{j,k}^{-1}\rho_{j,g-1}^{-1}\rho_{k,g}\quad \text{by }(\textnormal{b}_{1})\\
	&=& 
	\rho_{k,g}^{-1}\rho_{k,g-1}^{-1}B_{j,k}\rho_{k,g-1}\rho_{k,g}\quad \text{by }(\textnormal{d}_{2})\\
	&=& a_{k}^{-1}B_{j,k}a_{k}
    \end{eqnarray*}
    and we get
    \begin{eqnarray*}
        a_{n}\cdots a_{1}B_{j,k}a_{1}^{-1}\cdots a_{n}^{-1} & = & 
	a_{n}\cdots a_{k+1}a_{k}a_{k-1}\cdots 
	a_{j+1}a_{k}^{-1}B_{j,k}a_{k}a_{j+1}^{-1}\cdots 
	a_{k-1}^{-1}a_{k}^{-1}a_{k+1}^{-1}\cdots a_{n}^{-1}\\
         & = & a_{n}\cdots a_{k+1}B_{j,k}a_{k+1}^{-1}\cdots 
	 a_{n}^{-1}\quad\text{by }(\textnormal{g})\\
	 &=& B_{j,k}\quad\text{by }(\textnormal{d}_{1}).
    \end{eqnarray*}
        
    \medskip
    \textit{(6)} By $(d_{1})$, $a_{1}=\rho_{1,g-1}\rho_{1,g}$ commutes with $B_{i,j}$ for $2\leq 
    i<j\leq n$. Thus, relation $(\textnormal{i})$ is a direct 
    consequence of $(\textnormal{h})$.
    
    \medskip
    \textit{(7)} A direct consequence of $(\textnormal{h})$ since $T_{1}=B_{1,2}\cdots B_{1,n}$.

\end{proof}

\subsection{The pure braid group $\boldsymbol{P_{n}(N_{g})}$ is residually 2-finite}

Following~\cite{GG1}, one has, for $g\geq 2$, a split exact sequence
\begin{equation}\label{eq:Splitexactsequence}
\xymatrix{
1 \ar[r] & P_{n-1}(N_{g,1})\ar[r]^{\mu} & P_{n}(N_{g})\ar[r]^-{\lambda} & 
P_{1}(N_{g})=\pi_{1}(N_{g})\ar[r] & 1}
\end{equation}
where $\lambda$ is induced by the map which forgets all strands except the first one, and $\mu$ is 
defined by capping the boundary component by a disc with one marked point (the first strand in 
$P_{n}(N_{g})$).  According to  the definition of $\mu$ and to Theorem~\ref{thm:PresentationCasBord}, 
$\im\mu$ is generated by $\{\rho_{i,k},\ 2\leq i\leq n,\ 
1\leq k\leq g\}\cup\nolinebreak[4]\{B_{i,j},\ 2\leq\nolinebreak[4] i<\nolinebreak[4]j\leq\nolinebreak[4] n\}$.

The section given in~\cite{GG1} is geometric, \emph{i.e.} it is induced by a crossed section at the level of
fibrations. In order to study the action 
of $\pi_{1}(N_{g})$ on $P_{n-1}(N_{g,1})$, we need an algebraic one. 
Recall that $\pi_{1}(N_{g})$ has a group presentation with generators 
$p_{1},\ldots,p_{g}$ and the single relation $p_{1}^{2}\cdots 
p_{g}^{2}=1$. We define the set map  $\sigma:\pi_{1}(N_{g})\to P_{n}(N_{g})$ by setting
\[
\sigma(p_{i})=\begin{cases}
\rho_{1,i} & \text{for } 1\leq i\leq g-3,\\
\rho_{1,g-2}U^{-1} & \text{for } i=g-2,\\
U\rho_{1,g-1} & \text{for } i=g-1,\\
\rho_{1,g}T_{1}^{-1} & \text{for } i=g.
\end{cases}
\]

\begin{prop}\label{prop:section}
    The map $\sigma$ is a well defined  homomorphism satisfying 
    $\lambda\circ\sigma=\id_{\pi_{1}(N_{g)}}$.
\end{prop}

\begin{proof}
    Since $\lambda(\rho_{1,i})=p_{i}$ for all $1\leq i\leq g$ and 
    $\lambda(U)=\lambda(T_{1})=1$, one has clearly 
    $\lambda\sigma=\id_{\pi_{1}(N_{g)}}$ if $\sigma$ is a group homomorpism. Thus, we 
    have just to prove that $\sigma(p_{1})^{2}\cdots\sigma(p_{g})^{2}=1$:
    \begin{eqnarray*}
	\sigma(p_{1})^{2}\cdots\sigma(p_{g})^{2} & = & 
	(\rho_{1,1}^{2}\cdots\rho_{1,g-3}^{2})(\rho_{1,g-2}U^{-1})^{2}(U\rho_{1,g-1})^{2}(\rho_{1,g}T_{1}^{-1})^{2}\\
	&=&\rho_{1,1}^{2}\cdots\rho_{1,g-3}^{2}\rho_{1,g-2}\underbracket[1pt][4pt]{U^{-1}\rho_{1,g-2}}\rho_{1,g-1}Ua_{1}T_{1}^{-1}\rho_{1,g}T_{1}^{-1}\\
	&=&\rho_{1,1}^{2}\cdots\rho_{1,g-3}^{2}\rho_{1,g-2}^{2}\underbracket[1pt][4pt]{U^{-1}\rho_{1,g-1}}Ua_{1}T_{1}^{-1}\rho_{1,g}T_{1}^{-1}\quad\text{by }(\textnormal{f}_{1})\\
	&=&\rho_{1,1}^{2}\cdots\rho_{1,g-3}^{2}\rho_{1,g-2}^{2}\rho_{1,g-1}U^{-1}T_{1}\underbracket[1pt][4pt]{Ua_{1}T_{1}^{-1}}\rho_{1,g}T_{1}^{-1}\quad\text{by }(\textnormal{f}_{2})\\
	&=&\rho_{1,1}^{2}\cdots\rho_{1,g-3}^{2}\rho_{1,g-2}^{2}\rho_{1,g-1}U^{-1}T_{1}T_{1}^{-1}Ua_{1}\rho_{1,g}T_{1}^{-1}\quad\text{by }(\textnormal{j})\\
	&=&\rho_{1,1}^{2}\cdots\rho_{1,g-3}^{2}\rho_{1,g-2}^{2}\rho_{1,g-1}a_{1}\rho_{1,g}T_{1}^{-1}\\
	&=&\rho_{1,1}^{2}\cdots\rho_{1,g-3}^{2}\rho_{1,g-2}^{2}\rho_{1,g-1}^{2}\rho_{1,g}^{2}T_{1}^{-1}\\
	&=&1\quad\text{by }(\textnormal{c}).
    \end{eqnarray*}
\end{proof}

So, the exact sequence~(\ref{eq:Splitexactsequence}) splits. In order to apply 
Theorem~\ref{thm:suite exacte}, we have to prove that the action of $\pi_{1}(N_{g})$ on 
$P_{n-1}(N_{g,1})$ is trivial on $H_{1}(P_{n-1}(N_{g,1});\FF_{2})$. This is the claim of the following 
proposition.

\begin{prop} \label{prop:commutator}
    For all $x\in\im\sigma$ and $a\in\im\mu$, one has 
    $[x^{-1},a^{-1}]=xax^{-1}a^{-1}\in\gamma_{2}^{2}(\im\mu)$.
\end{prop}

\begin{proof}
It is enough to prove the result for $a\in\{B_{j,k},\ 2\leq j<k\leq n\}\cup\{\rho_{j,l},\ 2\leq 
j\leq n\text{ and }1\leq l\leq g\}$ and $x\in\{\sigma(p_{1}),\ldots,\sigma(p_{g})\}$, respectively 
sets of generators of $\im\mu$ and $\im\sigma$. Suppose first that 
$2\leq j<k\leq n$. One has:

\begin{list}{$\bullet$}{\leftmargin15mm}
    \item $[\sigma(p_{i})^{-1},B_{j,k}^{-1}]=[\rho_{1,i}^{-1},B_{j,k}^{-1}]=1$ for $1\leq i\leq g-3$ by 
    $(\textnormal{d}_{1})$;
    \item$[\sigma(p_{g-2})^{-1},B_{j,k}^{-1}]=[U\rho_{1,g-2}^{-1},B_{j,k}^{-1}]=1$  by $(\textnormal{d}_{1})$ and 
    $(\textnormal{i})$;
    \item 
    $[\sigma(p_{g-1})^{-1},B_{j,k}^{-1}]=[\rho_{1,g-1}^{-1}U^{-1},B_{j,k}^{-1}]=1$  by $(\textnormal{d}_{1})$ and 
    $(\textnormal{i})$;
    \item 
    $[\sigma(p_{g})^{-1},B_{j,k}^{-1}]=[T_{1}\rho_{1,g}^{-1},B_{j,k}^{-1}]=1$  by $(\textnormal{d}_{1})$ and 
    $(\textnormal{k})$.
\end{list}

\medskip
Now, let $j$ and $l$ be integers such that $2\leq j\leq n$ and $1\leq 
l\leq g$ and let us first prove that 
$[\rho_{1,i}^{-1},\rho_{j,l}^{-1}]\in\gamma_{2}^{2}(\im\mu)$ for all 
$i$, $1\leq i\leq n$:

\begin{list}{$\bullet$}{\leftmargin15mm}
    \item this is clear for $i<l$ by $(\textnormal{b}_{1})$;
    \item for $i=l$, the relation $(\textnormal{b}_{2})$ gives 
    $[\rho_{1,l}^{-1},\rho_{j,l}^{-1}]=\rho_{j,l}^{-1}B_{1,j}^{-1}\rho_{j,l}$. But 
    $$B_{1,j}^{-1}=B_{2,j}\cdots B_{j-1,j}B_{j,j+1}\cdots 
    B_{j,n}\rho_{j,g}^{-2}\cdots\rho_{j,1}^{-2}\quad \text{(relation 
    }(\textnormal{c}))$$ is an element of 
    $\gamma_{2}^{2}(\im\mu)$ by $(\textnormal{e})$, thus we get 
    $[\rho_{1,l}^{-1},\rho_{j,l}^{-1}]\in\gamma_{2}^{2}(\im\mu)$.
    \item If $l<i$ then 
    $[\rho_{1,i}^{-1},\rho_{j,l}^{-1}]=[B_{1,j}\rho_{j,i}^{-1}B_{1,j}\rho_{j,i},\rho_{j,l}^{-1}]$ 
    by $(\textnormal{b}_{3})$ so $[\rho_{1,i}^{-1},\rho_{j,l}^{-1}]\in\gamma_{2}^{}(\im\mu)$ since $\rho_{j,l}$, 
    $\rho_{j,i}$ and $B_{1,j}$ are elements of $\im\mu$.
\end{list}
From this, we deduce the following facts.

\medskip
(1) $[\sigma(p_{i})^{-1},\rho_{j,l}^{-1}]\in\gamma_{2}^{2}(\im\mu)$ for $i\leq g-3$ since 
$\sigma(p_{i})=\rho_{1,i}$.

\medskip
(2) 
$[\sigma(p_{g-2})^{-1},\rho_{j,l}^{-1}]=[U\rho_{1,g-2}^{-1},\rho_{j,l}^{-1}]=
\rho_{1,g-2}[U,\rho_{j,k}^{-1}]\rho_{1,g-2}^{-1}[\rho_{1,g-2}^{-1},\rho_{j,l}^{-1}]$. But $U$ and 
$\rho_{j,l}^{-1}$ are elements of $\im\mu$ thus 
$[U,\rho_{j,l}^{-1}]\in\Gamma_{2}(\im\mu)\subset\gamma_{2}^{2}(\im\mu)$. Consequently, 
$\rho_{1,g-2}[U,\rho_{j,k}^{-1}]\rho_{1,g-2}^{-1}\in\gamma_{2}^{2}(\im\mu)$ since 
$\gamma_{2}^{2}(\im\mu)$ is a caracteristic subgroup of $\im\mu$ and $\im\mu$ is normal in 
$P_{n}(N_{g})$. Thus, we get $[\sigma(p_{g-2})^{-1},\rho_{j,l}^{-1}]\in\gamma_{2}^{2}(\im\mu)$.
In the same way, one has 
$$[\sigma(p_{g-1})^{-1},\rho_{j,l}^{-1}]=[\rho_{1,g-1}^{-1}U^{-1},\rho_{j,l}^{-1}]=
U[\rho_{1,g-1}^{-1},\rho_{j,l}^{-1}]U^{-1}[U,\rho_{j,l}^{-1}]\in\gamma_{2}^{2}(\im\mu).$$

\medskip
(4) At last, 
$$[\sigma(p_{g})^{-1},\rho_{j,l}^{-1}]=[T_{1}\rho_{1,g}^{-1},\rho_{j,l}^{-1}]=
\rho_{1,g}[T_{1},\rho_{j,l}^{-1}]\rho_{1,g}^{-1}[\rho_{1,g},\rho_{j,l}^{-1}]\in\gamma_{2}^{2}(\im\mu)$$
since $T_{1},\rho_{j,l}\in\im\mu$ and $[\rho_{1,g},\rho_{j,l}^{-1}]\in\gamma_{2}^{2}(\im\mu)$.
 
\end{proof}

We are now ready to prove the main result of this section

\begin{thm}\label{thm:closed}
    For all $g\geq 2$ and $n\geq 1$, the pure braid group $P_{n}(N_{g})$ is residually 2-finite.
\end{thm}

\begin{proof}
Proposition \ref{prop:section} says that the  sequence
        $$
    \xymatrix{
    1 \ar[r] & P_{n-1}(N_{g,1})\ar[r] & P_{n}(N_{g})\ar[r] & \pi_{1}(N_{g})\ar[r] & 1}
    $$
    splits. 
    Now $P_{n-1}(N_{g,1})$ is residually 2-finite (Theorem \ref{thm:casbord}). It is proved in~\cite{B1} 
    and~\cite{B2} that $\pi_{1}(N_{g})$ is residually free for $g\geq 4$, so it is residually 
    2-finite. This result is proved in~\cite{LM} (lemma~8.9) for 
    $g=3$. When $g=2$, $\pi_{1}(N_{2})$ has presentation $\langle 
    a,b\,|\,aba^{-1}=b^{-1}\rangle$ so is a 2-almost direct product of 
    $\Z$ by $\Z$. Since $\Z$ is residually 2-finite, $\pi_{1}(N_{2})$ 
    is residually 2-finite by corollary~\ref{cor:residually 
    p-finite}. So, using Proposition \ref{prop:commutator} and 
    Corollary~\ref{cor:residually p-finite},  we can conclude that 
    $P_{n}(N_{g})$ is residually 2-finite.
\end{proof}

\begin{rem}\label{rem:p}
It follows from the proof of  Theorem \ref{thm:closed} that, when $g>2$, if  $P_{n}(N_{g,1})$ is 
residually $p$-finite for some $p\not= 2$
then the pure braid group $P_{n}(N_{g})$ is also residually $p$-finite.
\end{rem}

\subsection{The case $\boldsymbol{P_{n}(\R\mathrm{P}^{2})}$}

The main reason to exclude  $N_{1}=\R\mathrm{P}^{2}$ in  Theorem \ref{thm:closed} is that the exact sequence~(\ref{eq:Splitexactsequence}) 
doesn't exist in this case, but forgetting at most $n-2$ strands we get  the following exact 
sequence ($1\leq m\leq n-2$ ; see~\cite{VB})
$$
\xymatrix{
1 \ar[r] & P_{m}(N_{1,n-m})\ar[r] & P_{n}(\R\mathrm{P}^{2})\ar[r] & 
P_{n-m}(N_{g})\ar[r] & 1}.
$$
This sequence splits if, and only if $n=3$ and $m=1$ (see~\cite{GG2}). Thus, what we 
know is the following:

\begin{list}{$\bullet$}{\leftmargin5mm}
    \item $P_{1}(\R\mathrm{P}^{2})=\pi_{1}(\R\mathrm{P}^{2})=\Z/2\Z$: $P_{1}(\R\mathrm{P}^{2})$ is 
    a 2-group.
    \item $P_{2}(\R\mathrm{P}^{2})=Q_{8}$, the quaternion group (see~\cite{VB}): 
    $P_{2}(\R\mathrm{P}^{2})$ is a 2-group.
    \item One has the exact sequence
    $$
    \xymatrix{
    1 \ar[r] & P_{1}(N_{1,2})\ar[r] & P_{3}(\R\mathrm{P}^{2})\ar[r] & 
    P_{2}(\R\mathrm{P}^{2})\ar[r] & 1}
    $$
    where $P_{1}(N_{1,2})=\pi_{1}(N_{1,2})$ is a free group of rank $2$, thus is residually 
    2-finite. Since $P_{2}(\R\mathrm{P}^{2})$ is 2-finite, we can conclude that 
    $P_{3}(\R\mathrm{P}^{2})$ is residually 2-finite using lemma 1.5 of~\cite{Gr}.
\end{list}


\section*{Appendix: a group presentation for $\boldsymbol{P_{n}(N_{g,b})}$}

\renewcommand{\thethmAnnexe}{{\Alph{thmAnnexe}}}\setcounter{thm}{0}

Here we apply classical method (see~\cite{B,GG3}) to give a 
presentation of the $n^{\textnormal{th}}$ pure braid group of a 
nonorientable surface with boundary. Since $b\geq 1$, we'll see 
$N_{g,b}$ as a disc $D^{2}$ with $g+b-1$ open discs removed and $g$ 
M\"{o}bius strips glued on $g$ boundary components so obtained 
(see figure~\ref{fig:generatorsBord}).

\def\svgwidth{0.5\textwidth}
\setlength{\unitlength}{\svgwidth}
\begin{figure}[h]
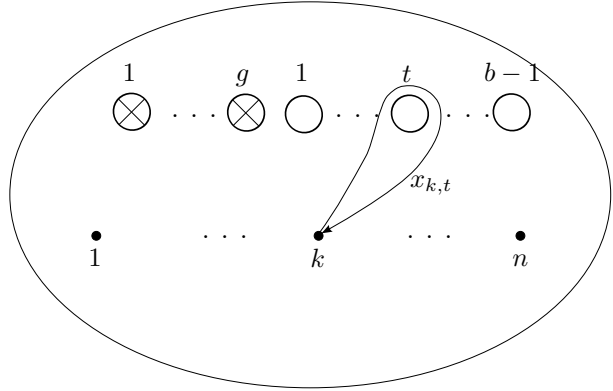

    \begin{center}
	\GenerateursBord
	\caption{Generators $x_{k,t}$ for $P_{n}(N_{g,b})$, $b\geq 1$}
	\label{fig:generatorsBord}
    \end{center}
\end{figure}

\begin{thmAnnexe}\label{thm:PresentationCasBord}
    For $g\geq 1$, $b\geq 1$ and $n\geq 1$, $P_{n}(N_{g,b})$ has the following 
    presentation:
    
    \medskip
    \noindent\ \underline{\bf generators:} $(B_{i,j})_{1\leq i<j\leq n}$, 
    $(\rho_{k,l})_{\build{}{\scriptstyle \!\!1\leq l\leq g}{\scriptstyle 1\leq k\leq n}}$ and 
    $(x_{u,t})_{\build{}{\scriptstyle 1\leq t\leq b-1}{\scriptstyle 
    \!\!\!\!\!\!1\leq u\leq n}}$.
    
    \medskip
    \noindent$\begin{array}[t]{lcl}
    \multicolumn{3}{l}{\underline{\bf relations:}\textnormal{ 
    (a)}\text{ for all }1\leq 
    i<j\leq n\textnormal{ and }1\leq r<s\leq n,}\\
    B_{r,s}B_{i,j}B_{r,s}^{-1}&=&
    \left\{\begin{array}{lll}
    B_{i,j} & \text{if }i<r<s<j\text{ or }r<s<i<j & (\textnormal{a}_{1})\\
    B_{i,j}^{-1} B_{r,j}^{-1}  B_{i,j} B_{r,j} B_{i,j} & \text{if 
    $r<i=s<j$}& (\textnormal{a}_{2})\\
    B_{s,j}^{-1} B_{i,j} B_{s,j} & \text{if $i=r<s<j$}& (\textnormal{a}_{3})\\
    B_{s,j}^{-1}B_{r,j}^{-1} B_{s,j} B_{r,j} B_{i,j} B_{r,j}^{-1} 
    B_{s,j}^{-1} B_{r,j} B_{s,j}   & \text{if $r<i<s<j$}& (\textnormal{a}_{4})
    \end{array}\right.\\
    \multicolumn{3}{l}{}\\
    \multicolumn{3}{l}{\hspace*{\longueur}
    \textnormal{ (b)}\text{ for all 
    }1\leq i<j\leq n\text{ and }1\leq k,l\leq g,}\\
    \rho_{i,k}\rho_{j,l}\rho_{i,k}^{-1}&=&
    \left\{\begin{array}{p{\longueurA}p{\longueurB}l}
    $\rho_{j,l}$ & $\text{if }k<l$ & (\textnormal{b}_{1})\\
    $\rho_{j,k}^{-1} B_{i,j}^{-1}  \rho_{j,k}^2$ & $\text{if }k=l$ & 
    (\textnormal{b}_{2})\\
    $\rho_{j,k}^{-1} B_{i,j}^{-1}\rho_{j,k} B_{i,j}^{-1} \rho_{j,l}
    B_{i,j} \rho_{j,k}^{-1} B_{i,j} \rho_{j,k}$ & $\text{if }k>l$ & 
    (\textnormal{b}_{3})
    \end{array}\right.\\
    \multicolumn{3}{l}{}\\
    \multicolumn{3}{l}{\hspace*{\longueur}
    \textnormal{ (d)}\text{ for all 
    }1\leq i<j\leq n,\ 1\leq k\leq n,\ k\neq j\text{ and }1\leq l\leq g,}\\
    \rho_{k,l}B_{i,j}\rho_{k,l}^{-1}&=&
    \left\{\begin{array}{p{\longueurA}p{\longueurB}l}
    $B_{i,j}$ & $\text{if }k<i\text{ or }j<k$ & (\textnormal{d}_{1})\\
    $\rho_{j,l}^{-1} B_{i,j}^{-1} \rho_{j,l}$ & $\text{if }k=i$ & 
    (\textnormal{d}_{2})\\
    $\rho_{j,l}^{-1} B_{k,j}^{-1} \rho_{j,l} B_{k,j}^{-1} B_{i,j} 
    B_{k,j} \rho_{j,l}^{-1} B_{k,j} \rho_{j,l}$ & $\text{if }i<k<j$ & 
    (\textnormal{d}_{3})
    \end{array}\right.\\
    \multicolumn{3}{l}{}\\
    \multicolumn{3}{l}{\hspace*{\longueur}
    \textnormal{ (l)}\text{ for all 
    }1\leq i<j\leq n,\ 1\leq u\leq n,\ u\neq j\text{ and }1\leq t\leq 
    b-1,}\\
    x_{u,t}B_{i,j}x_{u,t}^{-1}&=&
    \left\{\begin{array}{p{\longueurA}p{\longueurB}l}
    $B_{i,j}$ & $\text{if }u<i\text{ or }j<u$ & (\textnormal{l}_{1})\\
    $x_{j,t}^{-1} B_{i,j}x_{j,t}$ & $\text{if }u=i$ & 
    (\textnormal{l}_{2})\\
    $x_{j,t}^{-1} B_{u,j}x_{j,t} B_{u,j}^{-1} B_{i,j} 
    B_{u,j}x_{j,t}^{-1} B_{u,j}^{-1}x_{j,t}$ & $\text{if }i<u<j$ & 
    (\textnormal{l}_{3})
\end{array}\right.\\
    \multicolumn{3}{l}{}\\
    \multicolumn{3}{l}{\hspace*{\longueur}
    \textnormal{ (m)}\text{ for all 
    }1\leq k,u\leq n,\ k\neq u,\ 1\leq l\leq g\text{ and }1\leq t\leq b-1,}\\
    x_{u,t}\rho_{k,l}x_{u,t}^{-1}&=&
    \left\{\begin{array}{p{\longueurA}p{\longueurB}l}
    $\rho_{k,l}$ & $\text{if }k<u$ & (\textnormal{m}_{1})\\
    $x_{k,t}^{-1}B_{u,k}x_{k,t} B_{u,j}^{-1}\rho_{k,l}
    B_{u,k}x_{k,t}^{-1}B_{u,k}^{-1}x_{k,t}$ & $\text{if }u<k$ & 
    (\textnormal{m}_{2})
\end{array}\right.\\
    \multicolumn{3}{l}{}\\
    \multicolumn{3}{l}{\hspace*{\longueur}
    \textnormal{ (n)}\text{ for all 
    }1\leq i<j\leq n\text{ and }1\leq s,t\leq b-1,}\\
    x_{i,t}x_{j,s}x_{i,t}^{-1}&=&
    \left\{\begin{array}{p{\longueurA}p{\longueurB}l}
    $x_{j,s}$ & $\text{if }t<s$ & (\textnormal{n}_{1})\\
    $x_{j,t}^{-1} B_{i,j}x_{j,t}B_{i,j}^{-1}x_{j,t}$ & $\text{if }t=s$ & 
    (\textnormal{n}_{2})\\
    $x_{j,t}^{-1} B_{i,j}x_{j,t}B_{i,j}^{-1}x_{j,s}
    B_{i,j}x_{j,t}^{-1} B_{i,j}^{-1}x_{j,t}$ & $\text{if }s<t$ & 
    (\textnormal{n}_{3})
    \end{array}\right.\\
\end{array}$
\end{thmAnnexe}

\begin{proof}
    The proof works by induction and generalizes those of~\cite{GG3} 
    (closed non orientable case) and~\cite{B} (orientable case, possibly with boundary components). It uses the following short 
    exact sequence obtained by forgetting the last strand (see~\cite{FN}):
    \[
    \xymatrix{
    1 \ar[r] & 
    \pi_{1}(N_{g,b}\setminus\{z_{1},\ldots,z_{n}\},z_{n+1})\ar[r]^-{\alpha} & P_{n+1}(N_{g,b})\ar[r]^{\beta} & 
    P_{n}(N_{g,b})\ar[r] & 1.}
    \]
    The presentation is correct for $n=1$: 
    $P_{1}(N_{g,b})=\pi_{1}(N_{g,b})$ is free on the $\rho_{1,l}$'s 
    and $x_{1,t}$'s for $1\leq l\leq g$ and $1\leq t\leq b-1$. 
    Suppose inductively that $P_{n}(N_{g,b})$ has the given 
    presentation. Then, observe that $\{B_{i,n+1}\,/\,1\leq i\leq 
    n\}\cup\{\rho_{n+1,l}\,/\,1\leq l\leq 
    g\}\cup\{x_{n+1,t}\,/\,1\leq t\leq b-1\}$ is a free generators 
    set of $\im\alpha$ and $(B_{i,j})_{1\leq i<j\leq n}$, 
    $(\rho_{k,l})_{\build{}{\scriptstyle \!\!1\leq l\leq g}{\scriptstyle 1\leq k\leq n}}$ and 
    $(x_{u,t})_{\build{}{\scriptstyle 1\leq t\leq b-1}{\scriptstyle 
    \!\!\!\!\!\!1\leq u\leq n}}$ are coset representative for the 
    considered generators of $P_{n}(N_{g,b})$. There are three types 
    of relations for $P_{n+1}(N_{g,b})$. The first one comes from the 
    relations in $\im\alpha$: there are none here, since this group 
    is free. The second type 
    comes from the relations in $P_{n}(N_{g,b})$: they lift to the 
    same relations in $P_{n+1}(N_{g,b})$. Finally, the third type 
    arrives by studying the action of $P_{n}(N_{g,b})$ on $\im\alpha$ 
    by conjugation. We leave to the reader to verify that this action 
    corresponds to the given relations.
\end{proof}

\begin{rem}
    What precedes proves that $P_{n+1}(N_{g,b})$ is a semidirect 
    product of the free group \linebreak[4]
    $\pi_{1}(N_{g,b}\setminus\{z_{1},\ldots,z_{n}\},z_{n+1})$ by 
    $P_{n}(N_{g,b})$. Therefore, by recurrence, we get that  $P_{n+1}(N_{g,b})$ is an iterated 
    semidirect 
    product of (finitely generated) free groups.
\end{rem}

\vspace{10pt}

\noindent PAOLO BELLINGERI, Universit\'{e} de Caen, CNRS UMR 6139, LMNO, Caen, 14000 (France).
Email: paolo.bellingeri@math.unicaen.fr

\vspace{5pt}
\noindent SYLVAIN GERVAIS, Universit\'{e} de Nantes, CNRS-UMR 6629, 
Laboratoire Jean Leray, 2, rue de la Houssini\`{e}re, F-44322 NANTES 
cedex 03 (France). Email: sylvain.gervais@univ-nantes.fr

\end{document}